\documentclass[reqno]{amsart}

\usepackage[english]{babel}

\usepackage{mathrsfs, mathtools, amssymb}

\usepackage{dsfont}

\usepackage[paper=a4paper, margin=3cm]{geometry}

\usepackage[breaklinks]{hyperref}       % hyperlinks
\usepackage{booktabs}       % professional-quality tables
\usepackage{amsfonts}       % blackboard math symbols
\usepackage{nicefrac}       % compact symbols for 1/2, etc.
\usepackage{microtype}      % microtypography
\usepackage{lipsum}

\usepackage{amsmath,amssymb,amsthm,amscd}
\usepackage{tikz,tikz-cd}
\usetikzlibrary{arrows,arrows.meta}

\usepackage{graphicx}
\graphicspath{}
\DeclareGraphicsExtensions{.pdf,.png,.jpg,.jpeg}

\usepackage{yhmath}
\usepackage{enumitem}
\usepackage{breakurl}
\usepackage{caption}
\usepackage{cleveref}
\usepackage{float}

\crefname{figure}{Figure.}{Figures.}
\Crefname{figure}{Figure.}{Figures.}

\newcounter{stmcounter}[section]

\newtheorem{formula}{}[section]
\newtheorem{proposition}[formula]{Proposition}

\newtheorem{lemma}[formula]{Lemma}
\newtheorem{theorem}[formula]{Theorem}

\theoremstyle{definition}

\newtheorem{definition}[formula]{Definition}

\theoremstyle{remark}

\newtheorem{remark}[formula]{Remark}

\renewcommand{\d}{\text{d}}

\tikzcdset{arrow style=tikz, diagrams={>=stealth}}

\catcode`,\active

\catcode`\,12

\begin{document}
	
	\title[Geometrize the nongeometric space]{Geometrize the nongeometric space}

	\author{Tran Chi Quy$^\dagger$}
	\thanks{$^\dagger$Email: 23000008@hus.edu.vn}

    \author{Sonnet Nguyen Quang Hung$^*$}
	\thanks{$^*$Email: sonnet3001@gmail.com}

	\dedicatory{Faculty of Physics, VNU University of Science, Hanoi 100000, Vietnam}
	
	
\begin{abstract}
Noncommutative and nonassociative geometry has attracted much attention in both mathematics and physics communities. A prominent example is a stringy $R$-space over an underlying manifold, which exhibits both noncommutative and nonassociative structures and is known to be locally nongeometric. However, previous studies have mainly described $R$-space through algebraic relations of local string position coordinates, instead of a precise global geometric definition. In this paper, we propose a rigorous and systematic mathematical formulation of nongeometric $R$-space within the framework of descent theory. The central idea is to regard nongeometric $R$-space as carrying only local data of a quasitriangular quasi-Hopf algebra on each local coordinate domain of the underlying manifold, which are then glued coherently together to produce a global geometric object. We show that the appropriate global structure is not an ordinary space or a globally defined algebra, but a bicategory-valued stack. This provides a global and geometric definition of $R$-space while still preserving its locally nongeometric intrinsic character.
\end{abstract}

	\maketitle
    \tableofcontents
	\section*{Introduction}
    Understanding the geometry of spacetime at the quantum scale remains one of the central problems in modern theoretical physics. While general relativity describes gravity as the curvature of a pseudo-Riemannian manifold, quantum theory suggests that this geometric picture should break down at sufficiently small distances, where vacuum fluctuations are not compatible with Riemannian geometry. A consistent quantization of gravity therefore requires not only a quantum theory of the gravitational field, but also a deeper understanding of the spacetime geometry itself, which may no longer be described by the Riemannian geometry of classical gravity. Many theories have been introduced to solve the problem of quantum gravity, among which string theory stands out as the most promising candidate. Various studies in string theory have provided diverse perspectives on the geometry of spacetime at the quantum scale, revealing a range of novel mathematical structures that extend far beyond the traditional framework of Riemannian manifolds envisioned by most physicists. These alterations have, in turn, motivated deep and fruitful connections between string theory and several branches of pure mathematics. However, certain quantum geometric structures arising in string theory, such as \textit{noncommutative geometry} and \textit{nonassociative geometry}, remain conceptually and mathematically elusive, particularly concerning precise definitions. In this paper, we aim for a better understanding of the mathematical structures underlying these quantum geometries.
    
    After the success of noncommutative geometry in physics, nonassociative geometry has received attention in the physics community as a potential candidate for quantized spacetime \cite{PhysRevD.62.081501,nesterov2019nonassociative}. Within the framework of string theory, noncommutative spacetime is known to arise from modifications of the boundary conditions of \textit{open strings} on the boundary of the string worldsheet in the presence of a Neveu–Schwarz two-form field \cite{seiberg1999string,blumenhagen2014course}. Subsequent studies have further shown that analogous noncommutative geometric features can also emerge in the theory of \textit{closed string}, along with even more intricate structures, such as nonassociative geometry. The study of closed string theory is essential for stringy quantum gravity, since the excitation spectrum of closed strings contains massless bosonic fields, consisting of the spacetime metric $g$, the Neveu–Schwarz two-form field $B$, and the dilaton field $\Phi$. In other words, closed string theory encodes not only the data of spacetime itself but also the geometric structure of quantum spacetime. The noncommutativity and nonassociativity emerge from closed string flux compactifications with \textit{nongeometric fluxes}. More explicit, nongeometric fluxes can be realized from T-duality along the cycles of a higher-dimensional tori $\mathbb{T}^n$ (with $n>2$) endowed with a nonvanishing Neveu-Schwarz three-form flux $H=\d B$ by the following T-duality chain
   \[\begin{tikzcd}
       H_{ijk} \arrow{r}{T_i} & {f^i}_{jk} \arrow{r}{T_j} & {Q^{ij}}_k \arrow{r}{T_k} & {R^{ijk}}
    \end{tikzcd}\;,
    \]
    where $T_i$ (with $i=1,2,3$) denotes the T-duality transformations along the $i$-th cycle. Backgrounds that contain $Q$-flux or $R$-flux are nongeometric, respectively called $Q$-space and $R$-space (see \cite{plauschinn2019non,shelton2005nongeometric} for more detail). To be more specific, the nongeometry refers to the mixing of the original and dual coordinates, together with momentum and winding modes, due to the appearance of $Q$-flux or $R$-flux arising from T-duality transformations. Explicit calculations in string theory have shown that closed strings, which can wind and propagate in the cotangent bundle $\mathcal{M}:=T^*M$ of the spacetime $M$, effectively experience a noncommutative and nonassociative spacetime through the nonvanishing commutator and associator of space coordinate functions that can be determined via nongeometric fluxes. Analogous phenomena can also be found in condensed matter physics. Noncommutative geometry in condensed matter physics \cite{bellissard1994noncommutative,prodan2016bulk,du2024noncommutative} emerges when there are magnetic-like objects, such as external magnetic fields, Berry curvature, rotating frame, etc. On the other hand, nonassociative geometry in condensed matter physics represents the existence of objects similar to magnetic monopoles \cite{szabo2018magnetic}, which can be found in some materials, such as quantum spin ice \cite{morris09monopole}. For our interests, we focus on studying the $R$-space in this paper because of its fully nongeometric properties, including nonassociativity. However, we emphasize that our results can apply to other noncommutative and nonassociative spaces that are similar to $R$-space without necessarily being placed within the framework of string theory.

    The $R$-space is known to have no locally geometric description in terms of ordinary Riemannian geometry, since all the point-like objects on $R$-space are forbidden \cite{plauschinn2019non}. As we have just explained, the dynamics of closed strings in a spacetime $M$ on a background with $R$-flux are equivalent to an effectively noncommutative and nonassociative geometry. In an explicitly and systematically mathematical sense, these geometries are actually algebras over a coordinate domain of an underlying manifold, which can be encompassed in three-dimensional topological field theory, with the target space being a Courant algebroid \cite{roytenberg2007aksz,ikeda2003chern}. The perturbation expansion of such a membrane sigma model motivates a description of the noncommutative and nonassociative geometry within the framework of deformation quantization \cite{mylonas2014nonassociative} (also see \cite{moshayedi2022kontsevich} for Kontsevich's product). From these studies, explicit noncommutative and nonassociative products between two functions on a coordinate domain of $R$-space were obtained via Kontsevich’s formalism or, alternatively, via Drinfeld twist of the universal enveloping algebra into a quasitriangular quasi-Hopf algebra with a $\hbar$-adic topology. However, a conceptual gap remains in the existing literature. Most approaches to noncommutative and nonassociative geometry still begin with, or rely on, a globally defined quasitriangular quasi-Hopf algebra acting on a fixed space of functions or modules. In other words, previous studies of nongeometric $R$-space only considered the formalism of deformation quantization, including star products and twist, on a local coordinate domain $U \subset \mathcal{M}$ and assumed it to be true across the entire $\mathcal{M}$ via the global Hopf algebra $\mathscr{H}$, which is fixed across the entire manifold $\mathcal{M}$. Therefore, we realize that fixing a Lie algebra $\mathfrak{g}$ or fixing a Hopf algebra $\mathscr{H}$ as the universal enveloping algebra of Lie algebra $\mathfrak{g}$ over entire target space $\mathcal{M}$ is an overly strong and unnatural assumption for a nongeometric background with $R$-flux.

    Motivated by this gap, our paper introduces a global and geometric definition of the locally nongeometric $R$-space. Intuitively, we start from "small enough" regions of the target space such that within that region, a description of the $R$-space by a single Hopf algebra is acceptable. Such regions are called \textit{admissible}, in which each admissible region is assigned to a Hopf algebra. Consequently, a $R$-space over each admissible region contains the local data of the algebraic structures, including the noncommutativity and nonassociativity. Then, a $R$-space over a region that is too "large" or over the entire manifold $\mathcal{M}$ is understood as the suitable gluing of all local data. The ideas of gluing local data together and using higher structures to describe nongeometry, T-duality, and related topics have appeared in \cite{nikolaus2020higher,kim2022non}, where the authors studied nongeometric T-dualities (also see \cite{alfonsi2021puzzle} for the idea of gluing Double Field Theory patches). In mathematics, such processes are realized as \textit{descent theory}. 

    Descent—piecing together a global picture out of local pieces and gluing data—has appeared frequently in Grothendieck's philosophy and work, starting with his papers \cite{GrothendieckBourbakiDescent}. Descent theory plays an important role in many areas of pure mathematics, including algebraic geometry, algebraic topology, category theory, etc. The idea of descent theory consists of a category called a \textit{site}, a \textit{Grothendieck topology} to make the objects of a site act like open sets of a topological space, a \textit{fibered category} over a \textit{site}, local descent data, and descent conditions that ensure the local descent data can be glued together into a global object. A fibered category that satisfies descent conditions for all possible gluings is known as a \textit{stack}. In mathematics, a stack is a sheaf that takes values in categories rather than sets. Therefore, stacks can be realized as a generalization of fibered manifolds over a topological space. To achieve a stacky description, one can define the Grothendieck topology on a site, which consists of a covering family $\mathcal{U}=\{U_i\longrightarrow U\}_{i\in I}$, where $U_i$ are objects in a site. Then, one defines a pseudofunctor $\mathscr{P}$, called \textit{prestack}, that assigns to each object $U_i$ of a site, which plays the role of an open set, an object in $\mathscr{P}(U_i)$, and an isomorphism in $\mathscr{P}(U_i \times_\mathcal{U} U_j)$, even higher isomorphisms over fiber product of more than two objects of the site for higher stack. A prestack, together with its descent data, that satisfies the descent conditions becomes a stack. The process that makes a prestack become a stack is called \textit{stackification}. 

    In this research, we present a systematic construction of a global and geometric definition of the locally nongeometric $R$-space by using descent theory and stackification. Within the framework of deformation quantization, we show that the properties of $R$-space, including noncommutativity and nonassociativity, are encoded into quasitriangular quasi-Hopf algebras and their representation categories by Drinfeld's twist. We also show that the quantum geometry of $R$-space is compatible with the theory of a bicategory-valued prestack over a site of admissible coordinate domains of the target space $\mathcal{M}$, instead of a category-valued prestack due to $\hbar$-adic topology. Following descent theory, we stackify the bicategory-valued prestack into a well-defined bicategory-valued stack that is not just compatible with admissible coordinate domains, but with any open set of $\mathcal{M}$. This process is equivalent to gluing local descent data into a global object over an arbitrary union of admissible domains, which may be non-admissible. 

The paper is organized as follows. In Section \ref{sec:Preliminaries}, we review quantized universal enveloping algebras, quasitriangular quasi-Hopf algebras, and their behavior under twisting. Section \ref{sec:Representation category} studies the representation category of a quasitriangular quasi-Hopf algebra and its braided monoidal structure. Section \ref{sec:Stackification} contains the main results, where we construct the bicategory-valued prestack associated with $R$-space, define its 2-stackification, and prove the corresponding bicategorical descent theorem. Finally, Section \ref{sec:Admissible coordinate domains} provides physical arguments for the existence of admissible coordinate domains.

%%%%%%%%%%%%%%%%%%%%%%%%%%%%%%%%%%%%%%%%%%%%%%%%%%%%%%%%%%%%%%%%%%%%%%%%%%%%%%%%%%%
\section*{Summary}

In this work, we have developed a higher-categorical framework for the global geometric description of $R$-space over a manifold $\mathcal{M}$ without assuming the existence of a globally defined Hopf algebra. Our main result is the construction shown in Section \ref{sec:Stackification} and all the following proofs. Unlike conventional approaches, which typically start from a fixed global algebraic structure, our construction is based on local quasitriangular quasi-Hopf algebras assigned to admissible coordinate domains and their corresponding representation categories. We have shown that a category-valued prestack description is generally insufficient for capturing the global structure of $R$-space. The essential obstruction arises from the non-strictness of the transition data on triple overlaps. In the presence of the $\hbar$-adic topology, the compatibility of local Hopf-algebra data is governed by an infinite tower of equations, making a strict 1-categorical description unnatural. This observation motivates the passage from a category-valued prestack to a bicategory-valued prestack obtained by delooping the representation categories. Within this bicategorical framework, monoidal natural transformations appear as coherence 2-morphisms that encode the failure of strictness in local transition data. We further established that descent data associated with different admissible covering families are related by biequivalences induced from common refinements. These biequivalences define an equivalence relation among descent data and provide a covering-independent description of the global object. Based on these results, we constructed the 2-stackification of the bicategory-valued prestack and proved that it satisfies the bicategorical descent condition. Consequently, $R$-space admits a well-defined description as a bicategory-valued stack. Our result is, in some sense, analogous to the result of \cite{nikolaus2020higher}, in which the authors showed that T-backgrounds admit a description as a 2-stack. However, the two constructions differ in their underlying descent data: in \cite{nikolaus2020higher}, the 2-stack is glued from descent data consisting of a principal $\mathbb{T}^n$-bundle and a bundle gerbe, while in our work, the global stacky description of $R$-space is obtained by gluing descent data formed by representation categories of twisted universal enveloping algebras. In this formulation, the global geometry of $R$-space is not encoded by a single global Hopf algebra but rather by a coherent system of local representation categories together with higher transition data, which can be written down in a simple definition (see Theorem \ref{theorem:stackification}):
\begin{equation*}
\boxed{
R\text{-space over underlying manifold }\mathcal{M} =\dfrac{\coprod\limits_{\mathcal{U}\in\text{ACov}(\mathcal{M})} {\text{Des}_\mathscr{P}(\mathcal{U})}}{\sim_{\text{ref,bieq}}} = \mathop{2\text{-colim}}\limits_{\mathcal{U}\in\text{ACov}(\mathcal{M})}{\text{Des}_\mathscr{P}(\mathcal{U})},
}
\end{equation*}
where $\mathscr{P}$ is a pseudofunctor that assigns to each admissible coordinate domain a representation bicategory, $\text{ACov}(\mathcal{M})$ is the category of admissible covering families of manifold $\mathcal{M}$, and ${\text{Des}_\mathscr{P}(\mathcal{U})}$ is the category of $\mathscr{P}$-descent data relative to a covering family $\mathcal{U}$.

Thus, our construction provides a geometric realization of locally nongeometric $R$-space in the language of descent theory and higher geometry. It establishes a bridge between noncommutative and nonassociative structures arising from quasitriangular quasi-Hopf algebras and the theory of higher stacks. More broadly, our construction can provide a foundation for extending geometric tools, such as cohomological methods and Riemann-Roch type theorems, to a class of locally nongeometric quantum spacetimes motivated by theoretical physics. However, we leave it for future studies. In this perspective, the stack-theoretic formulation of $R$-space suggests a way to study quantum spacetime not as a classical manifold equipped with a single global algebra of functions, but as a coherent system of local quantum symmetries, higher categorical transition data, and fibered categories. This may offer a deeper geometric understanding of spacetime structures at the quantum level.

%%%%%%%%%%%%%%%%%%%%%%%%%%%%%%%%%%%%%%%%%%%%%%%%%%%%%%%%%%%%%%%%%%%%%%%%%%%%%%%%%%%
    \section{Preliminaries}\label{sec:Preliminaries}
    In this section, we revisit the theory of cocycle twist and its cochain extended version, within the framework of deformation quantization. We mainly follow \cite{Dr3,Drinfeld1990QuasiHopf,Drinfeld1991Quasitriangular,majid2000foundations} and references therein. As we mentioned before, the existence of a nonvanishing NS two-form field in string theory can lead to noncommutativity of the space coordinates, which can be realized as a finite-dimensional Lie algebra $\mathfrak{g}$ with the following commutator relation \cite{seiberg1999string,plauschinn2019non,blumenhagen2014course}
    \begin{equation}
        [x^\mu,x^\nu] = i\hbar \theta^{\mu\nu},
    \end{equation}
where $x^\mu, x^\nu$ are two different coordinates of a coordinate domain, $\theta$ is called noncommutative parameter. Consequently, the product of two differential functions is deformed into a noncommutative star product or even a nonassociative one, known by physicists as the Moyal star product and Kontsevich's product. In the framework of deformation quantization, we show that the algebra of functions in the noncommutative and nonassociative scenario can be described as a quasitriangular quasi-Hopf algebra, which is controlled by a 2-cochain twist and an associator that refers to the NS two-form field and its coboundary.

\subsection{Quantized universal enveloping algebra and formal $\hbar$-adic topology}

We start from the finite-dimensional Lie algebra $\mathfrak{g}$ over $\mathbb{C}$. The universal enveloping algebra $U(\mathfrak{g})$ is defined as the quotient of the tensor algebra $T(\mathfrak{g}) = \displaystyle\bigoplus\limits_{k \geq 0} \mathfrak{g}^{\otimes k}$ by the two-side ideal generated by $\mathcal{I}=x\otimes y - y\otimes x -[x,y]$ with $x,y \in \mathfrak{g}$. The algebra $U(\mathfrak{g})$ has a natural Hopf algebra structure with the following symmetric coalgebra structure
\begin{align}
    \Delta &: U(\mathfrak{g}) \longrightarrow U(\mathfrak{g}) \otimes U(\mathfrak{g}) \,, & \Delta(x) &= x \otimes 1 + 1 \otimes x \,, \nonumber \\[1ex]
\varepsilon &: U(\mathfrak{g}) \longrightarrow \mathbb{C} \,, & \varepsilon(x) &= 0 \,, \\[1ex]
S &: U(\mathfrak{g}) \longrightarrow U(\mathfrak{g}) \,, & S(x) &= -x,\nonumber
\end{align}
which is defined on the generators $x \in U(\mathfrak{g})$ and the unit element $1$. Here $\Delta$ is the coproduct, $S$ is the antipode, and $\varepsilon$ is the counit. This algebra is realized as an infinite-dimensional Hopf algebra and is denoted by $\mathscr{H}$. In fact, a general infinite-dimensional Hopf algebra $\mathscr{H}$ might not consist of a well-defined dual Hopf-algebra $\mathscr{H}^*$ because of the inequality $\mathscr{H}^* \otimes \mathscr{H}^*  \subseteq \left(\mathscr{H}\otimes \mathscr{H}\right)^*$. However, we can extend the notion of the universal enveloping algebra to formal power series in $\hbar$, which defines a Hopf algebra $U(\mathfrak{g})[[\hbar]]$ over $\mathbb{C}[[\hbar]]$ as an algebra complete in what is called $\hbar$-adic topology \cite{Dr3}.

\begin{definition}
    Let $\mathscr{H}$ be a topologically free $\mathbb{C}[[\hbar]]$-module and $\mathscr{H}/\hbar\mathscr{H} = U(\mathfrak{g})$ is an (infinite-dimensional) universal enveloping algebra of the Lie algebra $\mathfrak{g}$. Then $\mathscr{H}=U(\mathfrak{g})[[\hbar]]$ is an (infinite-dimensional) Hopf algebra over $\mathbb{C}[[\hbar]]$ with a symmetric coalgebra structure consisting of a $\hbar$-adic topologically complete tensor product.
\end{definition}
The universal enveloping algebra $U(\mathfrak{g})[[\hbar]]$ is an infinite-dimensional Hopf algebra with a natural \textit{invertible antipode}. 

\subsection{Quasitriangular Hopf algebra from cocycle twist}
The Hopf algebra $\left(\mathscr{H},\Delta,S,\varepsilon\right)$ can be embedded into a wider class of \textit{quasitriangular Hopf algebras} by equipping an additional quasitriangular structure $\mathcal{R}_0 \in Z^2\left(\mathscr{H}^*,\mathbb{C}[[\hbar]]\right)$.
\begin{definition}
    A quasitriangular Hopf algebra $(\mathscr{H},\mathcal{R})$ is a Hopf algebra with an invertible 2-cocycle $\mathcal{R} \in \mathscr{H} \otimes \mathscr{H}$ such that
    \begin{equation}\label{eq1.def2.1}
       \left( {\Delta  \otimes {\text{id}}} \right)\mathcal{R} = {\mathcal{R}_{13}}{\mathcal{R}_{23}}, \qquad \qquad \left( {{\text{id}} \otimes \Delta } \right)\mathcal{R} = {\mathcal{R}_{13}}{\mathcal{R}_{12}},
    \end{equation}
    \begin{equation}\label{eq2.def2.1}
        \left( {\tau  \circ \Delta } \right)a = \mathcal{R}\left( {\Delta a} \right){\mathcal{R}^{ - 1}} \qquad \qquad \forall a \in \mathscr{H},
    \end{equation}
    where $\mathcal{R} = \sum\limits_i {\mathcal{R}_i^{\left( 1 \right)} \otimes \mathcal{R}_i^{\left( 2 \right)}} $ is the Sweedler notation and 
    \begin{equation}
        {\mathcal{R}_{mn}} = \displaystyle\sum\limits_i {1 \otimes  \ldots  \otimes\underbrace{\mathcal{R}_i^{\left( 1 \right)} }_{m^{\textrm{th}} \textrm{ factor}}\otimes 1 \otimes  \ldots  \otimes\underbrace{ \mathcal{R}_i^{\left( 2 \right)}} _{n^{\textrm{th}} \textrm{ factor}}\otimes  \ldots  \otimes 1} .
    \end{equation}
    Here $\tau$ denotes the transposition map.
\end{definition}

\begin{lemma}
    If $(\mathscr{H}, \mathcal{R})$ is a quasitriangular Hopf algebra, then $\mathcal{R} \in \mathscr{H}\otimes \mathscr{H}$ obeys
    \begin{equation}\label{eq1.lemma1}
        (\varepsilon\otimes \text{id})\mathcal{R}= (\text{id}\otimes \varepsilon)\mathcal{R}=1,
    \end{equation}
    then also 
    \begin{equation}
        (S\otimes \text{id})\mathcal{R}=\mathcal{R}^{-1}, \qquad \qquad (\text{id}\otimes S)\mathcal{R}^{-1}=\mathcal{R},
    \end{equation}
    and hence $(S\otimes S)\mathcal{R}=\mathcal{R}$.
\end{lemma}
In the general case, $(\mathscr{H},\mathcal{R})$ is a noncocommutative algebra, or more precisely, a quasicocommutative algebra. The term \textit{quasicocommutative} means the cocommutativity is held only up to conjugacy by (\ref{eq2.def2.1}). The 2-cocycle $\mathcal{R}$ with transposition operator $\tau$ defines a braided monoidal category. In addition, the Hopf algebra $\mathscr{H}$ can be realized as a trivial quasitriangular Hopf algebra $\left(\mathscr{H},\mathcal{R}_0\right)$ with $\mathcal{R}_0 = 1\otimes 1$. Starting from a trivial quasitriangular Hopf algebra, one can obtain a nontrivial quasitriangular Hopf algebra by a twist $\mathcal{T}$, which is an invertible counital 2-cocycle.

\begin{theorem}\label{theorem for new quasitriangular Hopf alg}
    Let $(\mathscr{H},\mathcal{R})$ a quasitriangular Hopf algebra and $\mathcal{T}$ is a invertible counital 2-cocycle. Then exists a new quasitriangular Hopf algebra $(\mathscr{H}_\mathcal{T}, \mathcal{R}_\mathcal{T})$ defined by the same algebra and counit. The new quasitriangular Hopf algebra satisfies
\begin{equation}\label{theorem 2.4}
        \mathcal{R}_\mathcal{T}=\mathcal{T}_{21} \mathcal{R} \mathcal{T}^{-1},\qquad \Delta_\mathcal{T}(a) = \mathcal{T} \Delta(a) \mathcal{T}^{-1}, \qquad S_\mathcal{T} (a)= US(a)U^{-1}, \qquad \forall a \in \mathscr{H}_\mathcal{T},
    \end{equation}
    where $\mathcal{T}_{21} =\sum\limits_i \mathcal{T}_i^{(2)}\otimes {T}_i^{(1)}$ and $U=\sum\limits_i \mathcal{T}_i^{(1)}S\left(\mathcal{T}_i^{(2)}\right)$ is invertible.
\end{theorem}
\begin{proof}
    First, it is clear that $\Delta_\mathcal{T}$ is an algebra map because conjugation by $\mathcal{T}$ is an algebra automorphism. $\Delta$ is already coassociative. Therefore, $\Delta_\mathcal{T}$ is coassociative. Next, we show that $\mathcal{R}_\mathcal{T}$ defines a quasitriangular structure for $\mathscr{H}_\mathcal{T}$ by showing $\mathcal{R}_\mathcal{T}$ satisfies Definition 2.1. 
    \begin{equation}
        \begin{split}
            (\Delta_{\mathcal{T}}\otimes \textrm{id})\mathcal{R}_{\mathcal{T}}
&= \mathcal{T}_{12}\bigl((\Delta\otimes \textrm{id})(\tau(\mathcal{T})\mathcal{R}\mathcal{T}^{-1})\bigr)\mathcal{T}_{12}^{-1} \\
&= \mathcal{T}_{12}((\Delta\otimes \textrm{id})\tau(\mathcal{T}))((\Delta\otimes \textrm{id})\mathcal{R})((\Delta\otimes \textrm{id})\mathcal{T}^{-1})\mathcal{T}_{12}^{-1} \\
&= \mathcal{T}_{12}((\Delta\otimes \textrm{id})\tau(\mathcal{T}))\mathcal{R}_{13}\mathcal{R}_{23}((\textrm{id}\otimes \Delta)\mathcal{T}^{-1})\mathcal{T}_{23}^{-1} \\
&= \mathcal{T}_{12}((\Delta\otimes \textrm{id})\tau(\mathcal{T}))\mathcal{R}_{13}((\textrm{id}\otimes \tau\circ \Delta)\mathcal{T}^{-1})\mathcal{R}_{23}\mathcal{T}_{23}^{-1} \\
&= \mathcal{T}_{31}(\mathcal{T}^{(1)}_{(2)}\otimes \mathcal{T}^{(2)}\otimes \mathcal{T}^{(1)}_{(1)})\mathcal{R}_{13}((\textrm{id}\otimes \tau\circ \Delta)\mathcal{T}^{-1})\mathcal{R}_{23}\mathcal{T}_{23}^{-1} \\
&= \mathcal{T}_{31}\mathcal{R}_{13}(\mathcal{T}^{(1)}_{(1)}\otimes \mathcal{T}^{(2)}\otimes \mathcal{T}^{(1)}_{(2)})((\textrm{id}\otimes \tau\circ \Delta)\mathcal{T}^{-1})\mathcal{R}_{23}\mathcal{T}_{23}^{-1} \\
&= \mathcal{T}_{31}\mathcal{R}_{13}\mathcal{T}_{13}^{-1}\mathcal{T}_{32}((\textrm{id}\otimes \tau\circ \Delta)\mathcal{T})((\textrm{id}\otimes \tau\circ \Delta)\mathcal{T}^{-1})\mathcal{R}_{23}\mathcal{T}_{23}^{-1} \\
&= \mathcal{T}_{31}\mathcal{R}_{13}\mathcal{T}_{13}^{-1}\mathcal{T}_{32}\mathcal{R}_{23}\mathcal{T}_{23}^{-1} 
= (\mathcal{R}_{\mathcal{T}})_{13}(\mathcal{R}_{\mathcal{T}})_{23},
        \end{split}
    \end{equation}
    and also
    \begin{equation}
        \begin{split}
            \tau \circ \Delta_\mathcal{T} (a) = \mathcal{T}_{21} (\tau \circ\Delta(a))\mathcal{T}_{21}^{-1} = \mathcal{T}_{21}\mathcal{R}(\Delta(a))\mathcal{R}^{-1} \mathcal{T}_{21}^{-1} = \mathcal{R}_\mathcal{T}\mathcal{T}(\Delta(a))\mathcal{T}^{-1}\mathcal{R}_\mathcal{T}^{-1} = \mathcal{R}_\mathcal{T}(\Delta_\mathcal{T}(a))\mathcal{R}_\mathcal{T}^{-1}.
        \end{split}
    \end{equation}
    Similarly for all the rest. Finally, we define $U=\mathcal{T}^{(1)} S\left(\mathcal{T}^{(2)}\right)$ and $U^{-1}=S\left(\bar{\mathcal{T}}^{(1)}\right) \bar{\mathcal{T}}^{(2)}$. One can easily check that $UU^{-1}=U^{-1}U=1$. We compute
    \begin{equation}
\begin{split}
 (S_{\mathcal{T}}\otimes \textrm{id})\Delta_{\mathcal{T}}(a)
&= U\bigl(S(\mathcal{T}^{(1)}a_{(1)}\bar{\mathcal{T}}^{(1)})\bigr)U^{-1}
   \mathcal{T}^{(2)}a_{(2)}\bar{\mathcal{T}}^{(2)} \\
&= U(S \bar{\mathcal{T}}^{(1)} )(S a_{(1)})(S\mathcal{T}^{(1)})
   (S\bar{\mathcal{T}}^{\prime(1)})\bar{\mathcal{T}}^{\prime(2)}
   \mathcal{T}^{(2)}a_{(2)}\bar{\mathcal{T}}^{(2)} \\
&= U(S \bar{\mathcal{T}}^{(1)}) (S a_{(1)}) a_{(2)} \bar{\mathcal{T}}^{(2)}
   = \varepsilon(a) U U^{-1} = \varepsilon(a).
\end{split}
\end{equation}
Similarly, $(\textrm{id}\otimes S_{\mathcal{T}})\Delta_{\mathcal{T}}(a)
=\varepsilon(a)$. Thus, $S_\mathcal{T}$ is an antipode for $\mathscr{H}_\mathcal{T}$.
\end{proof}

We introduce a new class of (noncocommutative) quasitriangular Hopf algebras $\left(\mathscr{H}_\mathcal{T}, \mathcal{R}_\mathcal{T}\right)$ constructed from (cocommutative) Hopf algebras $\left(\mathscr{H},\mathcal{R}_0=1\otimes 1\right)$, where $\mathcal{R}_\mathcal{T} = \mathcal{T}_{21}\mathcal{T}^{-1}=\mathcal{T}^{-2}$. This algebra induces a twist from a commutative product on a left $\mathscr{H}$-module algebra to a noncommutative product on a left $\mathscr{H}_\mathcal{T}$-module algebra. For example, the universal enveloping algebra acts on the algebra of smooth functions over a coordinate domain $U$, denoted as untwisted $\left(\mathscr{A},\mu\right)=\left(C^\infty(U),\mu\right)$ and twisted $\left(\mathscr{A}_*,\mu_*\right)=\left(C^\infty(U),\mu_*\right)$ respectively refers to commutative product and noncommutative product between functions. The twisted algebra $\left(\mathscr{A}_*,\mu_*\right)$ is well-known as Moyal star product. 

\subsection{Quasitriangular quasi-Hopf algebra from cochain twist} 
To describe the $R$-space, we extend the noncocommutative structure of quasitriangular Hopf algebras into a class of noncocommutative noncoassociative quasitriangular quasi-Hopf algebras. By equipping an additional 3-cocycle $\phi \in Z^2\left(\mathscr{H}^*,\mathbb{C}[[\hbar]]\right)$, the 2-cocycle conditions of $\mathcal{R}$ is modified, thus, $\mathcal{R}$ is only required to be a 2-cochain.
\begin{definition}\label{def:quasibialgebra}
    A quasi-bialgebra $(\mathscr{H},\Delta,\varepsilon, \phi)$ is unital associative algebra $(\mathscr{H},\mu,\eta)$ with a dual algebraic structure of coalgebra $(\mathscr{H},\Delta,\varepsilon)$. The coalgebra is not necessarily a coassociative coalgebra with an invertible 3-cocycle $\phi \in \mathscr{H}\otimes\mathscr{H}\otimes\mathscr{H}$ and satisfies the following identities
    \begin{equation}
        \begin{split}
            \left( {\text{id}  \otimes \Delta} \right)\Delta \left( a \right) = \phi \left( {\left( {\Delta \otimes \text{id} } \right)\Delta \left( a \right)} \right){\phi ^{ - 1}}\qquad\qquad\forall a \in \mathscr{H},\\
             \left[ {\left( {\textrm{id} \otimes \textrm{id} \otimes \Delta } \right)\left( \phi  \right)} \right]\left[ {\left( {\Delta  \otimes \textrm{id} \otimes \textrm{id}} \right)\left( \phi  \right)} \right] = \left( {1 \otimes \phi } \right)\left[ {\left( {\textrm{id} \otimes \Delta  \otimes \textrm{id}} \right)\left( \phi  \right)} \right]\left( {\phi  \otimes 1} \right),\\
              \left( {\varepsilon  \otimes \text{id}} \right)\Delta \left( a \right) = a = \left( {\text{id} \otimes \varepsilon } \right)\Delta \left( a \right) \qquad\qquad \forall a \in \mathscr{H},\\
               \left( {\text{id} \otimes \varepsilon  \otimes \text{id}} \right)\left( \phi  \right) =\left( { \varepsilon  \otimes \text{id}\otimes \text{id}} \right)\left( \phi  \right) =\left( {\text{id} \otimes\text{id}\otimes \varepsilon } \right)\left( \phi  \right) = 1\otimes 1.
        \end{split}
    \end{equation}
\end{definition}
\begin{definition}
    A quasi-Hopf algebra $(\mathscr{H}, \Delta, \varepsilon, S, \phi, \alpha, \beta)$ is a quasi-bialgebra with an antipode be a triple $(S,\alpha,\beta)$, such that $\alpha,\beta \in \mathscr{H}$, $S$ is bijective and satisfies the following identities
    \begin{equation}
        \begin{split}
            &\sum\limits_i {S\left( {{a^{(1)}_i}} \right)\alpha {a^{(2)}_i}}  = \varepsilon \left( a \right)\alpha \qquad \sum\limits_i {{a^{(1)}_i}\beta S\left( {{a^{(2)}_i}} \right)}  = \varepsilon \left( a \right)\beta \qquad\quad \forall a \in \mathscr{H},\\
        &\sum\limits_i {{\phi^{{(1)}}_i}\beta S\left( {{\phi^{{(2)}}_i}} \right)\alpha {\phi^{{(3)}}_i}}  = 1 = \sum\limits_j {S\left( {{\bar{\phi}^{{(1)}}_j}} \right)\alpha {\bar{\phi}^{{(2)}}_j}\beta S\left( {{\bar{\phi}^{{(3)}}_j}} \right)},
        \end{split}
    \end{equation}
    where the Sweedler notations are $\Delta(a)=a^{(1)} \otimes a^{(2)}$, $\phi  = {{\phi^{(1)}} \otimes {\phi^{(2)}} \otimes {\phi^{(3)}}} $, and ${\phi ^{ - 1}} =  {{\bar{\phi}^{{(1)}}} \otimes {\bar{\phi}^{{(2)}}} \otimes {\bar{\phi}^{{(3)}}}} $.
\end{definition}
In this algebra, the coassociativity is
required to hold only up to a 3-cocycle $\phi$. The appearance of a 3-cocycle $\phi$ modifies the axioms (\ref{eq1.def2.1}), which now only constrains $\mathcal{R}$ to be a 2-cochain. 
\begin{definition}\label{def:quasitriangular quasi-Hopf}
    A quasitriangular quasi-Hopf algebra is a quasi-Hopf algebra with an invertible counital 2-cochain $\mathcal{R} \in \mathscr{H}\otimes \mathscr{H} $ obeying
    \begin{equation}
        (\Delta \otimes \textrm{id})\mathcal{R} = \phi_{312} \mathcal{R}_{13}\phi_{132}^{-1}\mathcal{R}_{23}\phi, \qquad (\textrm{id}\otimes \Delta )\mathcal{R}=\phi_{231}^{-1}\mathcal{R}_{13}\phi_{213}\mathcal{R}_{12}\phi^{-1}.
    \end{equation}
\end{definition}
In case we also relax the 2-cocycle conditions of the twist $\mathcal{T}$, a new class of quasitriangular quasi-Hopf algebra appears, similar to Theorem \ref{theorem for new quasitriangular Hopf alg}.

\begin{theorem}\label{theorem:twist of quasitri quasi-Hopf}
    Let $(\mathscr{H},\Delta,\varepsilon,S,\phi,\mathcal{R},\alpha,\beta)$ be a quasitriangular quasi-Hopf algebra and let $\mathcal{T}\in \mathscr{H}\otimes \mathscr{H}$ be an invertible counital 2-cochain that satisfies the normalization condition $(\varepsilon \otimes \textrm{id})\mathcal{T}=( \textrm{id}\otimes \varepsilon)\mathcal{T}=1$. Then $(\mathscr{H}_\mathcal{T},\Delta_\mathcal{T},\varepsilon,S_\mathcal{T},\phi_\mathcal{T},\mathcal{R}_\mathcal{T},\alpha_\mathcal{T},\beta_\mathcal{T})$ is a quasitriangular quasi-Hopf algebra where
    \begin{equation}
        \begin{split}
            \Delta_\mathcal{T}(a)=\mathcal{T} \Delta(a)\mathcal{T}^{-1} \qquad \qquad \forall a \in \mathscr{H}, \\
            \phi_\mathcal{T} = \mathcal{T}_{23}\left((\textrm{id}\otimes\Delta) \mathcal{T}\right)\phi \left((\Delta \otimes \textrm{id})\mathcal{T}^{-1}\right)\mathcal{T}_{12}^{-1}, \qquad \qquad S_\mathcal{T}=S, \\
            \alpha_\mathcal{T}= \sum\limits_i \left(S\bar{\mathcal{T}}^{(1)}\right)\alpha \bar{\mathcal{T}}^{(2)}, \qquad \qquad\beta_\mathcal{T}= \sum\limits_i \mathcal{T}^{(1)}\beta \left(S {\mathcal{T}}^{(2)}\right).
        \end{split}
    \end{equation}
\end{theorem}
\begin{proof}
    See Appendix \ref{appendix 1}.
\end{proof}

\begin{remark}\label{rmk:quasitri quasi-Hopf alg from twist}
    If we consider a Hopf algebra $\mathscr{H}$, which can be understood as a trivial quasitriangular quasi-Hopf algebra $\left(\mathscr{H},\phi_0,\mathcal{R}_0,\alpha_0,\beta_0\right)$ with $\mathcal{R}_0=1\otimes 1$ and $\phi = 1\otimes 1 \otimes 1$, a nontrivial quasitriangular quasi-Hopf algebra $\left(\mathscr{H}_\mathcal{T},\phi_\mathcal{T},\mathcal{R}_\mathcal{T},\alpha_\mathcal{T},\beta_\mathcal{T}\right)$ can be obtained from a 2-cochain twist $\mathcal{T}$, with the following structures
\begin{equation}
   \begin{split}
        \mathcal{R}_\mathcal{T} = \mathcal{T}_{21}\mathcal{T}^{-1} \in C^2\left(\mathscr{H}^*_\mathcal{T},\mathbb{C}[[\hbar]]\right), \qquad \qquad \phi_\mathcal{T} = \text{d}\mathcal{T} \in B^3 \left(\mathscr{H}^*_\mathcal{T},\mathbb{C}[[\hbar]]\right),\\
        S_\mathcal{T} = S, \qquad\qquad \alpha_\mathcal{T} = \mu \circ (S\otimes \text{id})(\mathcal{T}^{-1}), \qquad \beta_\mathcal{T} = \mu \circ (\text{id} \otimes S)(\mathcal{T}) = \alpha_\mathcal{T}^{-1} 
   \end{split}
\end{equation}
This is what we mentioned as \textit{cochain twist}. 
\end{remark}

If we consider that the starting point is a quasitriangular quasi-Hopf algebra $\left(\mathscr{H}_\mathcal{T},\text{d}\mathcal{T},\mathcal{T}_{21}\mathcal{T}^{-1},\alpha_\mathcal{T},\beta_\mathcal{T}\right)$, which is actually the universal enveloping algebra, the algebra of smooth functions on a coordinate domain is a left $\mathscr{H}_\mathcal{T}$-module algebra $\left(\mathscr{A}_\star,\star\right)$ with the product $\mu_\star: \mathscr{A}_\star \otimes \mathscr{A}_\star \longrightarrow \mathscr{A}_\star$ that is noncommutative and nonassociative in common sense. By taking the product of three elements and then rebracketing, we can prove $(a\star b)\star c \neq a\star (b\star c)$.

\begin{proposition}
    For all $a,b,c \in (\mathscr{A}_\star,\star)$,
    \begin{equation}\label{nonassociative product}
        (a\star b)\star c = \left(\phi_\mathcal{T}^{(1)} \triangleright a\right) \star \left[\left(\phi_\mathcal{T}^{(2)} \triangleright b\right) \star\left(\phi_\mathcal{T}^{(3)} \triangleright c\right) \right],
    \end{equation}
    where $\phi_\mathcal{T} = \phi_\mathcal{T}^{(1)}\otimes \phi_\mathcal{T}^{(2)}\otimes \phi_\mathcal{T}^{(3)}$.
\end{proposition}

\begin{proof}
    Let untwisted algebra $(\mathscr{A}, \mu)$ be a left $\mathscr{H}$-module algebra. The star product $\star$ is defined using the inverse twist $\mathcal{T}^{-1} = \bar{f}^1 \otimes \bar{f}^2 \in \mathscr{H} \otimes \mathscr{H}$ as:
\begin{equation}
a \star b = \mu \left( \mathcal{T}^{-1} \triangleright (a \otimes b) \right) = (\bar{f}^1 \triangleright a)(\bar{f}^2 \triangleright b)
\end{equation}
The associator $\phi_\mathcal{T} = \phi_{\mathcal{T}}^{(1)} \otimes \phi_{\mathcal{T}}^{(2)} \otimes \phi_{\mathcal{T}}^{(3)}$ is defined by:
\begin{equation}
\phi_\mathcal{T} = (1 \otimes \mathcal{T}) (\text{id} \otimes \Delta)(\mathcal{T}) (\Delta \otimes \text{id})(\mathcal{T}^{-1}) (\mathcal{T}^{-1} \otimes 1)
\end{equation}
which is equivalent to the operator identity:
\begin{equation}
(\text{id} \otimes \Delta)(\mathcal{T}^{-1}) (1 \otimes \mathcal{T}^{-1}) \cdot \phi_\mathcal{T} = (\Delta \otimes \text{id})(\mathcal{T}^{-1}) (\mathcal{T}^{-1} \otimes 1)
\end{equation}
We compute $(a \star b) \star c$ by applying the definition of $\star$ twice:
\begin{align*}
\text{LHS}=(a \star b) \star c = \bar{g}^1 \triangleright (a \star b) \cdot \bar{g}^2 \triangleright c = \bar{g}^1 \triangleright [(\bar{f}^1 \triangleright a)(\bar{f}^2 \triangleright b)] \cdot \bar{g}^2 \triangleright c
\end{align*}
Using the property of the $H$-module action $h \triangleright (xy) = (h_{(1)} \triangleright x)(h_{(2)} \triangleright y)$:
\begin{equation}
\text{LHS}=(a \star b) \star c = (\bar{g}^1_{(1)} \bar{f}^1 \triangleright a) (\bar{g}^1_{(2)} \bar{f}^2 \triangleright b) (\bar{g}^2 \triangleright c)
\end{equation}
This corresponds to the operator $X = (\Delta \otimes \text{id})(\mathcal{T}^{-1}) (\mathcal{T}^{-1} \otimes 1)$ acting on $a \otimes b \otimes c$. The RHS of (\ref{nonassociative product}) is $(\phi_{\mathcal{T}}^{(1)} \triangleright a) \star [(\phi_{\mathcal{T}}^{(2)} \triangleright b) \star (\phi_{\mathcal{T}}^{(3)} \triangleright c)]$. Let $x = \phi_{\mathcal{T}}^{(1)} \triangleright a$, etc. The term $x \star (y \star z)$ expands as:
\begin{equation}
\text{RHS}=x \star (y \star z) = (\bar{g}^1 \triangleright x) (\bar{g}^2_{(1)} \bar{f}^1 \triangleright y) (\bar{g}^2_{(2)} \bar{f}^2 \triangleright z)
\end{equation}
This corresponds to the operator $Y = (\text{id} \otimes \Delta)(\mathcal{T}^{-1}) (1 \otimes \mathcal{T}^{-1})$ acting on $x\otimes y \otimes z$.
Applying this to the elements shifted by $\phi_\mathcal{T}$, the total operator is $Y \phi_\mathcal{T}$ acting on $a\otimes b \otimes c$.
From the definition of $\phi_\mathcal{T}$, we have $X = Y  \phi_\mathcal{T}$. Therefore:
\begin{equation}
\text{LHS}=(a \star b) \star c = \mu^{(3)} (X \triangleright (a \otimes b \otimes c)) = \mu^{(3)} (Y  \phi_\mathcal{T} \triangleright (a \otimes b \otimes c)) = x \star (y \star z) = \text{RHS},
\end{equation}
where $\mu^{(3)} (a\otimes b\otimes c) = abc$ is the ordinary product of three elements in the untwisted algebra $(\mathscr{A},\mu)$. Under the cochain twist $\mathcal{T}$, the algebra $(\mathscr{A},\mu)$ becomes a twisted one $(\mathscr{A}_\star,\star)$.
\end{proof}

In summary, within the framework of an $R$-space, the existence of a noncommutative parameter and an $R$-flux corresponds to a 2-cocycle twist and a 3-coboundary, respectively.

%%%%%%%%%%%%%%%%%%%%%%%%%%%%%%%%%%%%%%%%%%%%%%%%%%%%%%%%%%%%%%%%%%%%%%%%%%%%%%%%%%%
\section{Representation category}\label{sec:Representation category}
In order to construct a geometric description of the $R$-space systematically, we are typically interested in the representations of the quasitriangular quasi-Hopf algebras that refer to the algebra of functions on a coordinate domain. The representation category of quasitriangular quasi-Hopf algebras has been studied in \cite{barnes2015nonassociative} and \cite[Chapter~9]{majid2000foundations} (also see \cite[Chapter~5]{etingof2015tensor} for category of finite representations). Thus, this section mainly repeats the results in \cite{barnes2015nonassociative}, which studied the category of representations of the quasi-Hopf algebra and its quasitriangular version. 

\subsection{Monoidal structure and braiding}
\begin{definition}
    A monoidal category is $\left(\mathcal{C},\otimes,\mathbb{I},\Phi,l,r\right)$, where $\mathcal{C}$ is a category with an unit object $\mathbb{I}$, $\otimes: \mathcal{C}\times \mathcal{C}\longrightarrow \mathcal{C}$ is an associative functor up to a natural equivalence $\Phi_{V,W,Z}:(V\otimes W)\otimes Z \mathop{\longrightarrow} \limits^{\cong}   V\otimes(W\otimes Z)$ obeying the pentagon condition (\cref{fig:pentagon condition}), and functorial isomorphisms $l_V:V \cong V\otimes\mathbb{I}$ and $r_V:V \cong \mathbb{I}\otimes V$ obeying the triangle condition (\cref{fig:triangle condition}) for all $V,W,Z \in \mathcal{C}$.
\end{definition}

\begin{figure}[ht]
\centering

\begin{minipage}{0.55\textwidth}
\centering
\begin{tikzpicture}[    
node distance=2.0cm and 2.4cm,     
obj/.style={font=\large},     
arrowlabel/.style={font=\small,inner sep=1pt},     
arrow/.style={-{To[length=7pt,width=9pt]},line width=0.6pt} ]

\node[obj](A) at (0,0) {$((U\otimes V)\otimes W)\otimes Z$};

\node[obj](B) at (2.4,1.6) {$(U\otimes V)\otimes (W\otimes Z)$};

\node[obj](C) at (4.8,0) {$U\otimes (V\otimes (W\otimes Z))$};

\node[obj](D) at (0,-1.8) {$(U\otimes (V\otimes W))\otimes Z$};

\node[obj](E) at (4.8,-1.8) {$U\otimes ((V\otimes W)\otimes Z)$};

\draw[arrow] (A) -- node[above left, arrowlabel] {$\Phi_{U\otimes V,W,Z}$} (B);

\draw[arrow] (B) -- node[above right, arrowlabel] {$\Phi_{U,V,W\otimes Z}$} (C);

\draw[arrow] (A) -- node[left, arrowlabel] {$\Phi_{U,V,W}\otimes \text{id}$} (D);

\draw[arrow] (D) -- node[below, arrowlabel] {$\Phi_{U,V\otimes W,Z}$} (E);

\draw[arrow] (E) -- node[right, arrowlabel] {$\text{id}\otimes \Phi_{V,W,Z}$} (C);

\end{tikzpicture}

\captionsetup{
    width=7.7cm,
    singlelinecheck=false
}
\caption{Pentagon condition for the associativity.}\label{fig:pentagon condition}
\end{minipage}
\hfill
\begin{minipage}{0.40\textwidth}
\centering
\begin{tikzpicture}[    
node distance=2.0cm and 2.4cm,     
obj/.style={font=\large},     
arrowlabel/.style={font=\small,inner sep=1pt},     
arrow/.style={-{To[length=7pt,width=9pt]},line width=0.6pt} ]

\node[obj](A) at (0,0) {$U\otimes V$};

\node[obj](B) at (-1.8,1.8) {$(U\otimes \mathbb{I})\otimes V$};

\node[obj](C) at (1.8,1.8) {$U\otimes (\mathbb{I}\otimes V)$};

\draw[arrow] (B) -- node[below left, arrowlabel] {$l_U\otimes \text{id}$} (A);

\draw[arrow] (C) -- node[below right, arrowlabel] {$\text{id}\otimes r_V$} (A);

\draw[arrow] (B) -- node[above, arrowlabel] {$\Phi_{U,\mathbb{I},V}$} (C);

\end{tikzpicture}
\captionsetup{
    width=4.5cm,
    singlelinecheck=false
}
\caption{Triangle condition for the compatibility with the unit object.}\label{fig:triangle condition}
\end{minipage}

\end{figure}

Let $\mathscr{H}$ be a quasitriangular quasi-Hopf algebra and $\text{Mod}_k$ be the category of $k$-modules, including the infinite-dimensional $k$-modules. We define the representation category $\text{Rep}(\mathscr{H})$ of $\mathscr{H}$ as the category of left $\mathscr{H}$-modules, which is also denoted as $^{\mathscr{H}}\text{Mod}$. The objects in $^{\mathscr{H}}\text{Mod}$ are pairs $V=\left(\underline{V}, { \triangleright _V}\right)$ of a $k$-module $\underline{V}$ and a $\mathscr{H}$-action ${ \triangleright _V}:\mathscr{H}\otimes \underline{V}\longrightarrow \underline{V}, h \otimes v \longmapsto h{ \triangleright _V} v$ satisfying the homomorphic algebra $h_1 { \triangleright _V}\left(h_2{ \triangleright _V}v\right) = (h_1h_2){ \triangleright _V}v$, $1{ \triangleright _V}v = v$ for all $v \in V,h\in\mathscr{H}$.

\begin{proposition}
    For any quasitriangular quasi-Hopf algebra, the category of representations of $\mathscr{H}$ is a braided monoidal category with a braiding structure
\begin{equation}\label{braiding structure}
\begin{array}{rccc}
\tau_{V,W}: & V\otimes W & \longrightarrow & W\otimes V \\
   & v\otimes w & \longmapsto  & \left(R^{(2)}\triangleright_W w\right)\otimes \left(R^{(1)}\triangleright_V v\right),
\end{array}
\end{equation}
for any objects $V,W$ in $^{\mathscr{H}}\text{Mod}$.
\end{proposition}

\begin{proof}
    For any objects $U,V,W$ in $^{\mathscr{H}}\text{Mod}$, let $\Phi_{U,V,W}:(U\otimes V)\otimes W \longrightarrow U\otimes (V\otimes W)$ be defined by the action of 3-cocycle $\phi= \phi^{(1)}\otimes \phi^{(2)}\otimes \phi^{(3)} \in Z^3\left(\mathscr{H},k\right) $
    \begin{equation}
        \Phi_{U,V,W} \left((u\otimes v)\otimes w \right) = \phi^{(1)}\triangleright u \otimes \left(\phi^{(2)}\triangleright v \otimes \phi^{(3)}\triangleright w\right),
    \end{equation}
    for all $u\in U$, $v\in V$, $w\in W$. According to the Definition \ref{def:quasibialgebra}, $\phi$ is invertible; thus, $\Phi_{U,V,W}$ is an isomorphism with an inverse given by the action of $\phi^{-1}$. Moreover, since the actions of $h \in \mathscr{H}$ on $(U\otimes V)\otimes W$ and $U\otimes (V\otimes W)$ are given by $(\Delta \otimes \text{id})\Delta(h)$ and $(\text{id}\otimes \Delta)\Delta(h)$, respectively; the first identity of the Definition \ref{def:quasibialgebra} also shows that $\Phi_{U,V,W}$ is $\mathscr{H}$-linear. 
    
    Firstly, we need to prove $\Phi$ satisfies the pentagon relation (\ref{fig:pentagon condition})
    \begin{equation}
        \Phi_{U,V,W\otimes Z} \circ \Phi_{U\otimes V,W, Z} = \left(\text{id}\otimes \Phi_{V,W,Z}\right)\circ \Phi_{U,V\otimes W,Z} \circ \left(\Phi_{U,V,W}\otimes \text{id}\right), 
    \end{equation}
    for any objects $U,V,W,Z$ in $^{\mathscr{H}}\text{Mod}$. The left-hand side is equal to the action 
    \begin{equation}
        \left[ {\left( {\textrm{id} \otimes \textrm{id} \otimes \Delta } \right)\left( \phi  \right)} \right]\left[ {\left( {\Delta  \otimes \textrm{id} \otimes \textrm{id}} \right)\left( \phi  \right)} \right] \triangleright ((U\otimes V)\otimes W)\otimes Z,
    \end{equation}
    while the right-hand side equals
    \begin{equation}
        \left( {1 \otimes \phi } \right)\left[ {\left( {\textrm{id} \otimes \Delta  \otimes \textrm{id}} \right)\left( \phi  \right)} \right]\left( {\phi  \otimes 1} \right) \triangleright ((U\otimes V)\otimes W)\otimes Z.
    \end{equation}
    From the 3-cocycle condition of $\phi$, both sides are equal. Thus, the functorial isomorphism $\Phi$ obeys the pentagon condition, which actually arises naturally from the 3-cocycle condition.

    Secondly, let us define the unit object as the trivial representation $\underline{1}=k$ and action $h\triangleright k = \varepsilon(h)k$. The triangle condition for the isomorphisms $l_V:V\cong V\otimes \underline{1}$ and $r_V:V\cong \underline{1}\otimes V$ is given by
    \begin{equation}
        (\text{id} \otimes r_V) \circ
\Phi_{U,k,V} = l_U \otimes \text{id}.   \end{equation}
The right-hand side is
\begin{equation}
    \text{RHS}= \left(l_U\otimes \text{id}\right)\left((u\otimes c)\otimes v\right) = cu\otimes v,
\end{equation}
and the left-hand side is
\begin{equation}
    \begin{split}
        \text{LHS}&=\phi^{(1)} \triangleright u \otimes \left(\varepsilon\left(\phi^{(2)}\right)c \otimes \phi^{(3)}\triangleright v\right)= \left(\text{id}\otimes \varepsilon \otimes \text{id}\right)(\phi)\triangleright(cu\otimes v) \\
        &=(1\otimes 1) (cu\otimes v)=cu\otimes v.
    \end{split}
\end{equation}
Here we use the normalization identity $\left(\text{id}\otimes \varepsilon \otimes \text{id}\right)(\phi)=1\otimes 1$. 

Finally, for the braiding structure (\ref{braiding structure}), which can be rewritten as $\tau_{V,W}=\sigma \circ R \triangleright_{V\otimes W}$, we can show that the two following hexagon diagrams commute for all objects in $^{\mathscr{H}}\text{Mod}$ by using Definition \ref{def:quasitriangular quasi-Hopf} and showing that both sides are equal, as we did with the previous coherence conditions. 

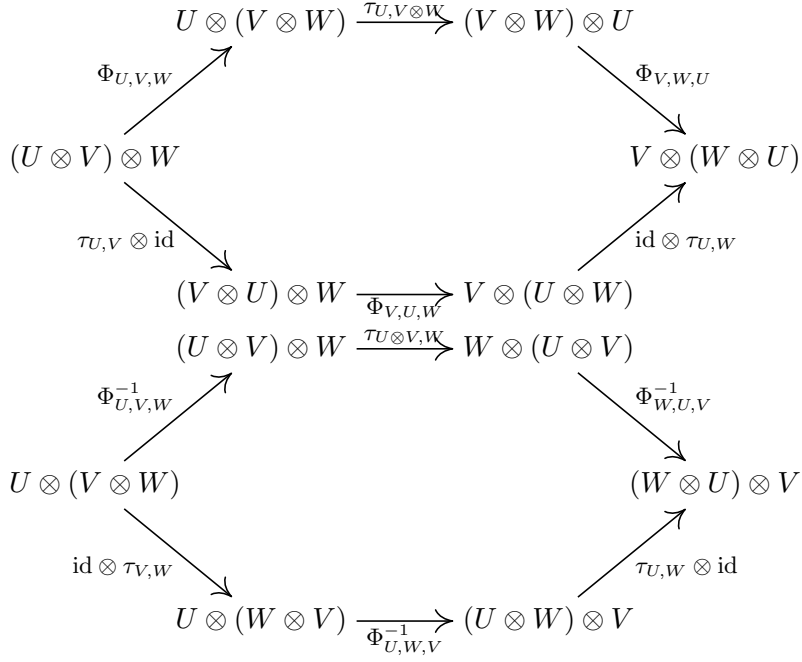
\begin{figure}[ht]
    \centering
    \begin{tikzpicture}[    
node distance=2.0cm and 2.4cm,     
obj/.style={font=\large},     
arrowlabel/.style={font=\small,inner sep=1pt},     
arrow/.style={-{To[length=7pt,width=9pt]},line width=0.6pt} ]

\node[obj](A) at (0,0) {$(U\otimes V)\otimes W$};

\node[obj](B) at (2.2,1.8) {$U\otimes (V\otimes W)$};

\node[obj](C) at (6,1.8) {$(V\otimes W) \otimes U$};

\node[obj](D) at (8.2,0) {$V\otimes (W\otimes U)$};

\node[obj](E) at (6,-1.8) {$V\otimes (U\otimes W)$};

\node[obj](F) at (2.2,-1.8) {$(V\otimes U)\otimes W$};

\draw[arrow] (A) -- node[above left, arrowlabel] {$\Phi_{U, V,W}$} (B);

\draw[arrow] (B) -- node[above, arrowlabel] {$\tau_{U,V\otimes W}$} (C);

\draw[arrow] (C) -- node[above right, arrowlabel] {$\Phi_{V,W,U}$} (D);

\draw[arrow] (A) -- node[below left, arrowlabel] {$\tau_{U,V}\otimes \text{id}$} (F);

\draw[arrow] (F) -- node[below, arrowlabel] {$\Phi_{V,U,W}$} (E);

\draw[arrow] (E) -- node[below right, arrowlabel] {$\text{id}\otimes \tau_{U,W}$} (D);
\end{tikzpicture}

\begin{tikzpicture}[    
node distance=2.0cm and 2.4cm,     
obj/.style={font=\large},     
arrowlabel/.style={font=\small,inner sep=1pt},     
arrow/.style={-{To[length=7pt,width=9pt]},line width=0.6pt} ]

\node[obj](A) at (0,0) {$U\otimes (V\otimes W)$};

\node[obj](B) at (2.2,1.8) {$(U\otimes V)\otimes W$};

\node[obj](C) at (6,1.8) {$W\otimes (U \otimes V)$};

\node[obj](D) at (8.2,0) {$(W\otimes U)\otimes V$};

\node[obj](E) at (6,-1.8) {$(U\otimes W)\otimes V$};

\node[obj](F) at (2.2,-1.8) {$U\otimes (W\otimes V)$};

\draw[arrow] (A) -- node[above left, arrowlabel] {$\Phi^{-1}_{U, V,W}$} (B);

\draw[arrow] (B) -- node[above, arrowlabel] {$\tau_{U\otimes V,W}$} (C);

\draw[arrow] (C) -- node[above right, arrowlabel] {$\Phi^{-1}_{W,U,V}$} (D);

\draw[arrow] (A) -- node[below left, arrowlabel] {$ \text{id}\otimes \tau_{V,W}$} (F);

\draw[arrow] (F) -- node[below, arrowlabel] {$\Phi^{-1}_{U,W,V}$} (E);

\draw[arrow] (E) -- node[below right, arrowlabel] {$ \tau_{U,W}\otimes \text{id}$} (D);
\end{tikzpicture}
\caption{Hexagon condition for the braiding structure.}
\end{figure}
The first diagram corresponds to identity
\begin{equation}
    (\textrm{id}\otimes \Delta )\mathcal{R}=\phi_{231}^{-1}\mathcal{R}_{13}\phi_{213}\mathcal{R}_{12}\phi^{-1},
\end{equation}
and the other one corresponds to 
\begin{equation}
     (\Delta \otimes \textrm{id})\mathcal{R} = \phi_{312} \mathcal{R}_{13}\phi_{132}^{-1}\mathcal{R}_{23}\phi,
\end{equation}
which coincide the Definition \ref{def:quasitriangular quasi-Hopf} of quasitriangular quasi-Hopf algebra. We completed the proof that $^{\mathscr{H}}\text{Mod}$ is a braided monoidal category.
\end{proof}

\begin{remark}
    The representation category of a quasitriangular quasi-Hopf algebra $\left(\mathscr{H}_\mathcal{T},\phi_\mathcal{T},\mathcal{R}_\mathcal{T},\alpha_\mathcal{T},\beta_\mathcal{T}\right)$ obtained from the cochain twist $\mathcal{T}$ (see Remark \ref{rmk:quasitri quasi-Hopf alg from twist}) is a braided monoidal category.
\end{remark}

\subsection{Category equivalence}
We recall that in our setup, the quasitriangular quasi-Hopf algebra $\mathscr{H}_\mathcal{T}$ that we are interested in is a cochain-twisted version of a Hopf algebra $\mathscr{H}$. Therefore, it is natural to determine the relationship between $\text{Rep}(\mathscr{H}):=^{\mathscr{H}}\text{Mod}$ and $\text{Rep}(\mathscr{H}_\mathcal{T}):=^{\mathscr{H}_\mathcal{T}}\text{Mod}$, which is equivalent to understand how $\text{Rep}(\mathscr{H})$ changes under the cochain twist $\mathcal{T}$ to become $\text{Rep}(\mathscr{H}_\mathcal{T})$. 

For any 2-cochain twist $\mathcal{T} \in \mathscr{H}\otimes \mathscr{H}$, any left $\mathscr{H}$-module is also a left $\mathscr{H}_\mathcal{T}$-module. Thus, for any objects $V,W$ in $\text{Rep}(\mathscr{H})$ with $\text{Rep}(\mathscr{H})$-morphism $f:V\longrightarrow W$, we define an identity functor 
\begin{equation}\label{eq:identity functor}
    \begin{array}{rccc}
{{\mathscr{F}}_{\mathcal{T}}}: & \text{Rep}(\mathscr{H})& \longrightarrow & \text{Rep}(\mathscr{H}_\mathcal{T}) \\
   & V & \longmapsto  & {{\mathscr{F}}_{\mathcal{T}}}(V),
\end{array}
\end{equation}
which also canonically induces a $\text{Rep}(\mathscr{H}_\mathcal{T})$-morphism ${{\mathscr{F}}_{\mathcal{T}}}(f):{{\mathscr{F}}_{\mathcal{T}}}(V)\longrightarrow{{\mathscr{F}}_{\mathcal{T}}}(W)$. Here, any object ${{\mathscr{F}}_{\mathcal{T}}}(V)$ is just the object $V$ when considered as an object in the representation category of $\mathscr{H}_\mathcal{T}$. The functor ${{\mathscr{F}}_{\mathcal{T}}}$ between the representation category of $\mathscr{H}$ and $\mathscr{H}_\mathcal{T}$ is invertible due to the invertible cochain twist $\mathcal{T}$. Hence, we have a category equivalence between the representation categories of $\mathscr{H}$ and $\mathscr{H}_\mathcal{T}$, which can also be realized from the intuition of $\hbar$-adic topological equivalence $\mathscr{H}/\hbar \mathscr{H}$. Furthermore, the functor ${{\mathscr{F}}_{\mathcal{T}}}$ is a monoidal functor that is compatible with the tensor product, which is given by natural $\text{Rep}(\mathscr{H}_\mathcal{T})$-isomorphisms
\begin{equation}
    \begin{array}{rccc}
J_{V,W}: & {{\mathscr{F}}_{\mathcal{T}}}(V)\otimes_\mathcal{T} {{\mathscr{F}}_{\mathcal{T}}}(W)& \longrightarrow & \mathscr{F(V\otimes W)} \\
   & v\otimes_\mathcal{T} w & \longmapsto  & \left(\bar{\mathcal{T}}^{(1)}\triangleright_V v\right)\otimes\left(\bar{\mathcal{T}}^{(2)}\triangleright_W w\right),
\end{array}
\end{equation}
for any objects $V,W$ in $\text{Rep}(\mathscr{H})$ and a twist $\mathcal{T}^{-1} = \bar{\mathcal{T}}^{(1)}\otimes\bar{\mathcal{T}}^{(2)}$, and 
\begin{equation}
    J_0: \underline{1}_\mathcal{T} \longrightarrow {{\mathscr{F}}_{\mathcal{T}}}(\underline{1}),
\end{equation}
for the unit object $\underline{1} = k$.

\begin{theorem}
    $\text{Rep}(\mathscr{H})$ and $\text{Rep}(\mathscr{H}_\mathcal{T})$ are equivalent as braided monoidal categories for any Hopf algebra $\mathscr{H}$ and any cochain twist $\mathcal{T} \in \mathscr{H}\otimes \mathscr{H}$.
\end{theorem}

\begin{proof}
    Using the normalization conditions in Definition \ref{def:quasitriangular quasi-Hopf}, it is easy to check that the following coherence diagrams commute
\begin{figure}[H]
\centering
\begin{minipage}{0.45\textwidth}
\centering
\begin{tikzpicture}[    
node distance=2.0cm and 2.4cm,     
obj/.style={font=\large},     
arrowlabel/.style={font=\small,inner sep=1pt},     
arrow/.style={-{To[length=7pt,width=9pt]},line width=0.6pt} ]

\node[obj](A) at (0,0) {${{\mathscr{F}}_{\mathcal{T}}}(V)\otimes_\mathcal{T}\underline{1}_\mathcal{T}$};

\node[obj](B) at (0,-1.6) {${{\mathscr{F}}_{\mathcal{T}}}(V)\otimes_\mathcal{T}{{\mathscr{F}}_{\mathcal{T}}}(\underline{1})$};

\node[obj](C) at (3.5,0) {${{\mathscr{F}}_{\mathcal{T}}}(V)$};

\node[obj](D) at (3.5,-1.6) {${{\mathscr{F}}_{\mathcal{T}}}(V\otimes \underline{1})$};

\draw[arrow] (A) -- node[left, arrowlabel] {$J_0\otimes_\mathcal{T}\text{id}$} (B);

\draw[arrow] (A) -- node[above, arrowlabel] {$l_{{{\mathscr{F}}_{\mathcal{T}}}(V)}$} (C);

\draw[arrow] (B) -- node[below, arrowlabel] {$J_{\underline{1},V}$} (D);

\draw[arrow] (D) -- node[right, arrowlabel] {${{\mathscr{F}}_{\mathcal{T}}}(l_V)$} (C);

\end{tikzpicture}
\end{minipage}
\hfill
\begin{minipage}{0.45\textwidth}
\centering
\begin{tikzpicture}[    
node distance=2.0cm and 2.4cm,     
obj/.style={font=\large},     
arrowlabel/.style={font=\small,inner sep=1pt},     
arrow/.style={-{To[length=7pt,width=9pt]},line width=0.6pt} ]

\node[obj](A) at (0,0) {$\underline{1}_\mathcal{T}\otimes_\mathcal{T}{{\mathscr{F}}_{\mathcal{T}}}(V)$};

\node[obj](B) at (0,-1.6) {${{\mathscr{F}}_{\mathcal{T}}}(\underline{1})\otimes_\mathcal{T}{{\mathscr{F}}_{\mathcal{T}}}(V)$};

\node[obj](C) at (3.5,0) {${{\mathscr{F}}_{\mathcal{T}}}(V)$};

\node[obj](D) at (3.5,-1.6) {${{\mathscr{F}}_{\mathcal{T}}}(\underline{1}\otimes V)$};

\draw[arrow] (A) -- node[left, arrowlabel] {$J_0\otimes_\mathcal{T}\text{id}$} (B);

\draw[arrow] (A) -- node[above, arrowlabel] {$r_{{{\mathscr{F}}_{\mathcal{T}}}(V)}$} (C);

\draw[arrow] (B) -- node[below, arrowlabel] {$J_{\underline{1},V}$} (D);

\draw[arrow] (D) -- node[right, arrowlabel] {${{\mathscr{F}}_{\mathcal{T}}}(r_V)$} (C);

\end{tikzpicture}
\end{minipage}
\end{figure}
for any object $V$ in $\text{Rep}(\mathscr{H})$. Next, from Theorem \ref{theorem:twist of quasitri quasi-Hopf}, we have
\begin{equation}
        \left((\textrm{id}\otimes\Delta) \mathcal{T}^{-1}\right)\left(1\otimes \mathcal{T}^{-1}\right)\phi_\mathcal{T} = \phi \left((\Delta \otimes \textrm{id})\mathcal{T}^{-1}\right)\left(\mathcal{T}^{-1}\otimes 1\right),
    \end{equation}
    which is equivalent to a commutative diagram of $\text{Rep}(\mathscr{H}_\mathcal{T})$-action consisting of the monoidal structure on $\text{Rep}(\mathscr{H}_\mathcal{T})$
    \begin{figure}[H]
    \centering
    \begin{tikzpicture}[    
node distance=2.0cm and 2.4cm,     
obj/.style={font=\large},     
arrowlabel/.style={font=\small,inner sep=1pt},     
arrow/.style={-{To[length=7pt,width=9pt]},line width=0.6pt} ]

\node[obj](A) at (0,0) {$\left({{\mathscr{F}}_{\mathcal{T}}}(U)\otimes_\mathcal{T} {{\mathscr{F}}_{\mathcal{T}}}(V)\right)\otimes_\mathcal{T}{{\mathscr{F}}_{\mathcal{T}}}(W)$};

\node[obj](B) at (9,0) {${{\mathscr{F}}_{\mathcal{T}}}(U)\otimes_\mathcal{T}\left({{\mathscr{F}}_{\mathcal{T}}}(V)\otimes_\mathcal{T} {{\mathscr{F}}_{\mathcal{T}}}(W)\right)$};

\node[obj](C) at (9,-2) {${{\mathscr{F}}_{\mathcal{T}}}(U)\otimes_\mathcal{T}{{\mathscr{F}}_{\mathcal{T}}}(V\otimes W)$};

\node[obj](D) at (9,-4) {${{\mathscr{F}}_{\mathcal{T}}}(U\otimes(V\otimes W))$};

\node[obj](E) at (0,-2) {${{\mathscr{F}}_{\mathcal{T}}}(U\otimes V)\otimes_\mathcal{T}{{\mathscr{F}}_{\mathcal{T}}}(W)$};

\node[obj](F) at (0,-4) {${{\mathscr{F}}_{\mathcal{T}}}((U\otimes V)\otimes W)$};

\draw[arrow] (A) -- node[above, arrowlabel] {$\Phi_{{{\mathscr{F}}_{\mathcal{T}}}(U),{{\mathscr{F}}_{\mathcal{T}}}(V),{{\mathscr{F}}_{\mathcal{T}}}(W)}$} (B);

\draw[arrow] (B) -- node[right, arrowlabel] {$\text{id}\otimes_\mathcal{T}J_{V,W}$} (C);

\draw[arrow] (C) -- node[right, arrowlabel] {$J_{U,V\otimes W}$} (D);

\draw[arrow] (A) -- node[left, arrowlabel] {$J_{U,V}\otimes_\mathcal{T} \text{id}$} (E);

\draw[arrow] (E) -- node[left, arrowlabel] {$J_{U\otimes V,W}$} (F);

\draw[arrow] (F) -- node[below, arrowlabel] {${{\mathscr{F}}_{\mathcal{T}}}(\Phi_{U,V,W})$} (D);
\end{tikzpicture}
\end{figure}
for any objects $U,V,W$ in $\text{Rep}(\mathscr{H})$. Finally, from the definition of the twisted $\mathcal{R}$-matrix in Theorem \ref{theorem for new quasitriangular Hopf alg}, we have the braiding
\begin{equation}
    \begin{split}
        J_{U,V} \left(\tau_{{{\mathscr{F}}_{\mathcal{T}}}(V),{{\mathscr{F}}_{\mathcal{T}}}(U)}(v\otimes_\mathcal{T}u)\right) = \tau_{V,U} \left(J_{V,U}(v\otimes_\mathcal{T}u)\right)\qquad \qquad \forall u\in U,v\in V,
    \end{split}
\end{equation}
which allows us to prove the following diagram of braiding commute.
\begin{figure}[H]
    \centering
    \begin{tikzpicture}[    
node distance=2.0cm and 2.4cm,     
obj/.style={font=\large},     
arrowlabel/.style={font=\small,inner sep=1pt},     
arrow/.style={-{To[length=7pt,width=9pt]},line width=0.6pt} ]

\node[obj](A) at (0,0) {${{\mathscr{F}}_{\mathcal{T}}}(V)\otimes_\mathcal{T} {{\mathscr{F}}_{\mathcal{T}}}(U)$};

\node[obj](B) at (6,0) {${{\mathscr{F}}_{\mathcal{T}}}(U)\otimes_\mathcal{T} {{\mathscr{F}}_{\mathcal{T}}}(V)$};

\node[obj](C) at (0,-2) {${{\mathscr{F}}_{\mathcal{T}}}(V\otimes U)$};

\node[obj](D) at (6,-2) {${{\mathscr{F}}_{\mathcal{T}}}(U\otimes V)$};

\draw[arrow] (A) -- node[above, arrowlabel] {$\tau_{{{\mathscr{F}}_{\mathcal{T}}}(V),{{\mathscr{F}}_{\mathcal{T}}}(U)}$} (B);

\draw[arrow] (B) -- node[right, arrowlabel] {$J_{U,V}$} (D);

\draw[arrow] (A) -- node[left, arrowlabel] {$J_{V,U}$} (C);

\draw[arrow] (C) -- node[below, arrowlabel] {${{\mathscr{F}}_{\mathcal{T}}}(\tau_{V,U})$} (D);
\end{tikzpicture}
\end{figure}
Therefore, the representation categories of $\mathscr{H}$ and $\mathscr{H}_\mathcal{T}$ are equivalent as braided monoidal categories.
\end{proof}

%%%%%%%%%%%%%%%%%%%%%%%%%%%%%%%%%%%%%%%%%%%%%%%%%%%%%%%%%%%%%%%%%%%%%%%%%%%%%%%%%%%
\section{Stackification}\label{sec:Stackification}
Our paper, so far, presents a category of representations of a quasitriangular quasi-Hopf algebra on a coordinate domain of the original manifold $\mathcal{M}$. In other words, each coordinate domain of $\mathcal{M}$ is assigned to a representation category of a quasitriangular quasi-Hopf algebra. Our goal is to provide a globally canonical description for locally nongeometric $R$-space. Therefore, we follow the idea of gluing in descent theory by constructing the global definition of an $R$-space from its local data, which is mentioned as \textit{sheafification} or \textit{stackification} process.

\subsection{Site and prestack}
First, from the observation that each coordinate domain of $\mathcal{M}$ is assigned to a representation category, we define the category of coordinate domains.
\begin{definition}
    Let $\mathcal{M}$ be a topological manifold. The category of admissible coordinate domains $\text{Loc}\mathcal{M}$ is a category that contains:
    \begin{enumerate}
        \item Objects are admissible open coordinate domains $U\subset \mathcal{M}$.
        \item Morphisms are inclusion maps
        \begin{equation}
            \begin{split}
                \text{Hom}_{\text{Loc}\mathcal{M}}(V,U) =\left\{ 
                \begin{array}{c l}
                     \left\{\iota_{VU}:V\hookrightarrow U\right\}, &\quad V\subset U, \\
                     \varnothing,  &\quad V \not\subset U,
                \end{array}\right.
            \end{split}
        \end{equation}
        for any objects $U,V$ in $\text{Loc}\mathcal{M}$.
    \end{enumerate}
\end{definition}
We will explain why we only choose admissible coordinate domains after the construction of the prestack.

\begin{proposition}
    The underlying topology on the topological manifold $\mathcal{M}$ naturally induces a Grothendieck topology on the category $\text{Loc}\mathcal{M}$.
\end{proposition}

\begin{proof}
    First, we identify the fibered product $U_1 \times_U U_2$ for any $U_1,U_2 \subset U$ with inclusion maps are the arrows $U_1 \longrightarrow U$ and $U_2 \longrightarrow U$. The following diagrams commute
    \begin{figure}[H]
\centering
\begin{minipage}{0.40\textwidth}
\centering
\begin{tikzpicture}[    
node distance=2.0cm and 2.4cm,     
obj/.style={font=\large},     
arrowlabel/.style={font=\small,inner sep=1pt},     
arrow/.style={-{To[length=7pt,width=9pt]},line width=0.6pt} ]

\node[obj] (A) at (0,0) {$U_1 \cap U_2$};

\node[obj] (B) at (0,-2) {$U_2$};

\node[obj] (C) at (2.4,0) {$U_1$};

\node[obj] (D) at (2.4,-2) {$U$};

\draw[arrow] (A) -- node[left, arrowlabel] {} (B);

\draw[arrow] (A) -- node[above, arrowlabel] {} (C);

\draw[arrow] (B) -- node[below, arrowlabel] {} (D);

\draw[arrow] (C) -- node[right, arrowlabel] {} (D);

\end{tikzpicture}
\end{minipage}
\hfill
\begin{minipage}{0.40\textwidth}
\centering
\begin{tikzpicture}[    
node distance=2.0cm and 2.4cm,     
obj/.style={font=\large},     
arrowlabel/.style={font=\small,inner sep=1pt},     
arrow/.style={-{To[length=7pt,width=9pt]},line width=0.6pt} ]

\node[obj] (A) at (0,0) {$W$};

\node[obj] (B) at (0,-2) {$U_2$};

\node[obj] (C) at (2.2,0) {$U_1$};

\node[obj] (D) at (2.2,-2) {$U$};

\draw[arrow] (A) -- node[left, arrowlabel] {} (B);

\draw[arrow] (A) -- node[above, arrowlabel] {} (C);

\draw[arrow] (B) -- node[below, arrowlabel] {} (D);

\draw[arrow] (C) -- node[right, arrowlabel] {} (D);

\end{tikzpicture}
\end{minipage}
\end{figure}
for any $W \longrightarrow U_1$, $W \longrightarrow U_2$; thus, $U_1 \cap U_2$ satisfies the universal property of the fibered product. Hence, 
\begin{equation}
      U_1 \times_U U_2\cong U_1 \cap U_2,
\end{equation}
for any objects $U_1 \subset U$ and $U_2 \subset U$. For any object $U$ in $\text{Loc}\mathcal{M}$, the family of arrows $\mathcal{U}=\left\{U_i\longrightarrow U\right\}_{i\in I}$ is a covering family and satisfies the axioms of Grothendieck topology. 
\begin{enumerate}
        \item Family of identity inclusion maps $\left\{\text{id}_U:U\longrightarrow U\right\}$ is a covering family.
        \item For any $V\subset U$, we have
        \begin{equation}
            \begin{split}
                \bigcup\limits_{i\in I}\left(U_i \cap V \right) = \left(\bigcup\limits_{i\in I} U_1\right) \cap V =U\cap V=V.
            \end{split}
        \end{equation}
        Thus, $\left\{U_i\cap V\longrightarrow V\right\}_{i\in I}$ is a covering family of $V \subset U$.
        \item For each $i \in I$, $\left\{V_{ij}\longrightarrow U_i\right\}_{j\in J_i}$ is a covering family of $U_i$. Because 
        \begin{equation}
            \bigcup\limits_{i\in I}\bigcup\limits_{j\in J_i} V_{ij} = \bigcup\limits_{i\in I} U_i =U,
        \end{equation}
        the family of composition arrows $\left\{V_{ij}\longrightarrow U_i \longrightarrow U\right\}_{i,j} = \left\{V_{ij}\longrightarrow U\right\}_{i,j}$ is a covering family.
\end{enumerate}
The set $J$ of the family of covers is a Grothendieck topology on category $\text{Loc}\mathcal{M}$.
\end{proof}
    
The category $\text{Loc}\mathcal{M}$, equipped with Grothendieck topology $J$, forms a \textit{site} $\mathscr{S}=\left(\text{Loc}\mathcal{M},J\right)$ covered by an open cover $\mathcal{U}=\left\{U_i\longrightarrow U\right\}_{i\in I}$. The fibered product between open domains can be realized as the intersection 
\begin{equation}
    U_{i_1} \times_\mathcal{U} U_{i_2}\times_\mathcal{U}\ldots \times_\mathcal{U}U_{i_n}  \cong  \bigcap\limits_{j=1}^n U_{i_j},   
\end{equation}
for any objects $\left\{U_{i_j}\right\}_{j=1}^n$ in $\mathscr{S}$.

Next, it is natural to introduce the 2-category $\textbf{CatRe}$, the 2-category of representation categories of quasitriangular quasi-Hopf algebras, where the objects are representation categories, 1-morphisms are functors between them, and 2-morphisms are natural transformations. We define the 1-prestack by the pseudofunctor
\begin{equation}
    \begin{split}
        \begin{array}{rccc}
\mathscr{P}_1: & \mathscr{S}^{\text{op}} & \longrightarrow & \textbf{CatRe} \\
   & U_i & \longmapsto  & \text{Rep}(\mathscr{H}_{U_i}),
\end{array}
    \end{split}
\end{equation}
where $\mathscr{H}_i$ is the quasitriangular quasi-Hopf algebra on an admissible coordinate domain $U_i$. The reason why we choose the site to consist only of admissible coordinate domains is that the local algebraic data of the theory are only well-defined on such domains. An admissible coordinate domain is an open set on which one can choose local coordinates and construct the corresponding local quasi-Hopf algebra, or equivalently, its representation category. In a more intuitive point of view, for a general open set $U\subset \mathcal{M}$, especially if $U$ is too large, there may be no single globally defined algebra $\mathscr{H}_U$. In the locally nongeometric situation, this failure is not an accident; it is part of the geometry. Therefore, assigning $U \longmapsto \mathscr{P}_1(U)=\text{BRep}(\mathscr{H}_U)$ for every open set $U$ would artificially assume the existence of global algebraic data, which we precisely want to avoid. Instead, we choose a site whose objects are admissible coordinate domains such that the local model exists on each such local domain. The key point is that a site does not need to contain all open sets. It is enough that admissible coordinate domains form a basis for the topology: every open set can be covered by admissible coordinate domains, and intersections can be refined by admissible coordinate domains. Thus, arbitrary open sets are not ignored; they are reconstructed by descent.

The morphisms, which are now the restriction map $\rho_{UV}:U \longrightarrow V$ adapting the composition $\rho_{UW}=\rho_{VW}\circ \rho_{UV}$, for all objects $W\subset V\subset U$ in $\mathscr{S}^{\text{op}}$. The restriction map $\rho_{UV}$ induces a restriction map $\mathscr{P}_1(\rho_{UV}):\text{Rep}(\mathscr{H}_{U})\longrightarrow \text{Rep}(\mathscr{H}_{V}), X\longmapsto X|_{V}$, which is compatible with the following commutative diagram.
\begin{figure}[H]
\centering
\begin{minipage}{0.50\textwidth}
\centering
\begin{tikzpicture}[    
node distance=2.0cm and 2.4cm,     
obj/.style={font=\large},     
arrowlabel/.style={font=\small,inner sep=1pt},     
arrow/.style={-{To[length=7pt,width=9pt]},line width=0.6pt} ]

\node[obj] (A) at (0,0) {$U$};

\node[obj] (B) at (0,-1.6) {$V$};

\node[obj] (C) at (3.5,0) {$\text{Rep}(\mathscr{H}_{U})$};

\node[obj] (D) at (3.5,-1.6) {$\text{Rep}(\mathscr{H}_{V})$};

\draw[arrow] (A) -- node[left, arrowlabel] {$\rho_{UV}$} (B);

\draw[arrow] (A) -- node[above, arrowlabel] {$\mathscr{P}_1$} (C);

\draw[arrow] (B) -- node[below, arrowlabel] {$\mathscr{P}_1$} (D);

\draw[arrow] (C) -- node[right, arrowlabel] {$\mathscr{P}_1(\rho_{UV})$} (D);

\end{tikzpicture}
\captionsetup{
    width=6cm,
    singlelinecheck=false
}
\caption{Morphisms on $\textbf{CatRe}$ as restriction maps.}
\end{minipage}
\hfill
\begin{minipage}{0.40\textwidth}
\centering
\begin{tikzpicture}[    
node distance=2.0cm and 2.4cm,     
obj/.style={font=\large},     
arrowlabel/.style={font=\small,inner sep=1pt},     
arrow/.style={-{To[length=7pt,width=9pt]},line width=0.6pt} ]

\node[obj] (A) at (0,0) {$\text{Rep}(\mathscr{H}_{U})$};

\node[obj] (B) at (3.6,0) {$\text{Rep}(\mathscr{H}_{V})$};

\node[obj] (C) at (1.8,-1.6) {$\text{Rep}(\mathscr{H}_{W})$};

\draw[arrow] (A) -- node[above, arrowlabel] {$\mathscr{P}_1(\rho_{UV})$} (B);

\draw[arrow] (A) -- node[below left, arrowlabel] {$\mathscr{P}_1(\rho_{UW})$} (C);

\draw[arrow] (B) -- node[below right, arrowlabel] {$\mathscr{P}_1(\rho_{VW})$} (C);

\end{tikzpicture}
\captionsetup{
    width=5cm,
    singlelinecheck=false
}
\caption{Composition of restriction maps on $\textbf{CatRe}$.}
\end{minipage}
\end{figure}

To construct a global geometric description, such as \textit{a stack}, one requires that the local objects be allowed to be patched together to become a global one. This step is called \textit{stackification}. For the 1-stackification, we follow the definitions in \cite{breen2009notes,vistoli2004notes,stacks-project}. The descent condition for a 1-stack is the gluing of the descent data on the double overlap $U_{ij}=U_i \cap U_j$ and triple overlap $U_{ijk}=U_i \cap U_j \cap U_k$. In this scenario, the \textit{descent datum} $\left(\mathcal{C}_i,g_{ij}\right)$ in category $\textbf{CatRe}$ relative to the family $\mathcal{U}=\left\{U_i\longrightarrow U\right\}_{i\in I}$ of open covers is given by an object $\mathcal{C}_i = \text{Rep}(\mathscr{H}_{U_i})$ and an isomorphism (category equivalence) $g_{ij}:\mathcal{C}_j|_{U_{ij}}\longrightarrow \mathcal{C}_i |_{U_{ij}}$ that preserves the braided monoidal structure of the representation categories of quasitriangular quasi-Hopf algebras. The descent condition on triple overlap $U_{ijk}$ requires $g_{ij} g_{jk} = g_{ik}$, which is called \textit{strict cocycle condition}, for every $(i,j,k)\in I^3$. A category-valued 1-stack is sufficient to glue local representation categories only when the equivalences satisfy a strict cocycle condition on triple overlaps. However, in our quasitriangular quasi-Hopf setting, the transitions of untwisted Hopf algebras $\left[\mathscr{H}\right]=\mathscr{H}|_{\mathcal{T}=1\otimes 1}$ may not adopt a strict cocycle condition on triple overlaps. For instance, the isomorphism $g_{ij}$ between representation categories induces a Hopf-algebra isomorphism $G_{ij}:\left[\mathscr{H}_j\right]\longrightarrow\left[\mathscr{H}_i\right]$ (see Proposition \ref{prop:gauge transformation} and Appendix \ref{appendix 2}). However, these underlying Hopf algebras actually are Hopf algebras over $\hbar$-adic topology, which should be manifested by $\mathscr{H}[[\hbar]]$ and $G_{ij}[[\hbar]]$. The algebra isomorphism $G_{ij}:=G_{ij}[[\hbar]]$ can be written in formal power series in $\hbar$
\begin{equation}
    G_{ij} = G_{ij}^{(0)}+\hbar G_{ij}^{(1)}+\mathcal{O}(\hbar^2),
\end{equation}
and emerges the composition
\begin{equation}
    G_{ij}G_{jk} = G_{ij}^{(0)}G_{jk}^{(0)}+\hbar \left(G_{ij}^{(0)}G_{jk}^{(1)}+G_{ij}^{(1)}G_{jk}^{(0)}\right) +\mathcal{O}(\hbar^2).
\end{equation}
The strict cocycle condition $G_{ij}G_{jk} = G_{ik}$ on any triple overlap $U_{ijk}$ is therefore an infinite tower of equations in powers of $\hbar$, which is not natural. Passing to representation categories, one requires a monoidal natural transformation 
\begin{equation}
    \alpha_{ijk}: g_{ij} g_{jk} \Longrightarrow g_{ik}
\end{equation}
to fix the $\hbar$-adic non-strictness of the cocycle condition on triple overlaps. The existence of $\alpha_{ijk}$ corresponds to a family of 2-arrows in 2-descent theory\cite{breen2009notes}. Therefore, we need to elevate our theory to a higher structure of 2-stackification, which describes fibered 2-categories over a site.

\subsection{Bicategory-valued prestack}
We reiterate that a local domain of a nongeometric $R$-space can be described by a representation category of a quasitriangular quasi-Hopf algebra fibered over the site $\mathscr{S}$, which is called a 1-prestack. In the previous subsection, we have shown that describing a $R$-space, in global sense, as a 1-stack is insufficient for gluing local objects into a global one due to the $\hbar$-adic non-strictness. Thus, we need a higher structure that consists of the monoidal natural transformations $\alpha_{ijk}: g_{ij} g_{jk} \Longrightarrow g_{ik}$. For this purpose, let us \textit{deloop} the category $\text{Rep}(\mathscr{H})$ into a one-object bicategory. 
\begin{definition}
    Let $\text{Rep}(\mathscr{H})$ be a braided monoidal category of representations of a quasitriangular quasi-Hopf algebra $\mathscr{H}$. The bicategory $\text{BRep}(\mathscr{H})$ is a one-object bicategory that consists of
    \begin{enumerate}
        \item There is only one object $*$,
        \item 1-morphisms are representations of $\mathscr{H}$,
        \item 2-morphisms are $\mathscr{H}$-linear maps between representations.
    \end{enumerate}
    Moreover, there is also a tricategory $\textbf{BCatRe}$ is a tricategory of bicategories of representations such that
    \begin{enumerate}
        \item Objects are bicategories $\text{BRep}(\mathscr{H})$,
        \item 1-morphisms are 2-functors,
        \item 2-morphisms are lax natural transformations,
        \item 3-morphisms are modifications,
    \end{enumerate}
\end{definition}
Functors between representation categories $\text{Rep}(\mathscr{H})$ automatically induced 2-functors between $\text{BRep}(\mathscr{H})$. Braided monoidal equivalence between categories of representations becomes braided monoidal equivalence between bicategories of representations. 

Following the previous step, we can also define a \textit{bicategory-valued prestack} (or \textit{weak 2-prestack}) by using the tricategory $\textbf{BCatRe}$ that is fibered over the site $\mathscr{S}$. The weak 2-prestack is given by the pseudofunctor
\begin{equation}
    \begin{split}
        \begin{array}{rccc}
\mathscr{P}: & \mathscr{S}^{\text{op}} & \longrightarrow & \textbf{BCatRe} \\
   & U_i & \longmapsto  & \text{BRep}(\mathscr{H}_{U_i}),
\end{array}
    \end{split}
\end{equation}
with the composition and restriction maps are the same as the previous construction of 1-prestack.

\subsection{2-descent conditions and transition on the same covering family}
In order to stackify the weak 2-prestack over the site $\mathscr{S}$, one requires the 2-descent data to satisfy the 2-descent conditions on all $n$-fold overlaps. Within the framework of weak 2-prestack (2-stack), the 2-descent datum becomes a 3-tuple instead of a 2-tuple in the 1-prestack (1-stack) cases.

\begin{definition}
    A 2-descent datum in category $\textbf{BCatRe}$ relative to the family $\mathcal{U}=\left\{U_i\longrightarrow U\right\}_{i\in I}$ of open covers is a 3-tuple $\left(\mathcal{C}_i,g_{ij},\alpha_{ijk}\right)$ that consists of
    \begin{enumerate}
        \item An object $\mathcal{C}_i = \text{BRep}(\mathscr{H}_{U_i})$ for each $i\in I$,
        \item A 1-arrow $g_{ij}:\mathcal{C}_j|_{U_{ij}}\longrightarrow\mathcal{C}_i|_{U_{ij}}$ for each pair $(i,j)\in I^2$,
        \item A 2-arrow $\alpha_{ijk}:g_{ij} g_{jk} \Longrightarrow g_{ik}$ for each $(i,j,k)\in I^3$,
    \end{enumerate}
    such that the family of 2-arrows forms a tetrahedral diagram of the restrictions to quadruple overlap of 2-arrows, whose each face is a triple overlap, commutes.
\end{definition}

\begin{figure}[H]
\centering
\begin{minipage}{0.35\textwidth}
\centering
\begin{tikzpicture}[    
node distance=2.0cm and 2.4cm,     
obj/.style={font=\large},     
arrowlabel/.style={font=\small,inner sep=1pt},     
arrow/.style={-{To[length=7pt,width=9pt]},line width=0.6pt},twoarrow/.style={-{Implies[length=7pt,width=9pt]},double,double distance=1.2pt,line width=0.4pt} ]

\node[obj] (A) at (0,0) {$\mathcal{C}_i$};

\node[obj] (B) at (2,2) {$\mathcal{C}_j$};

\node[obj] (C) at (4,0) {$\mathcal{C}_k$};

\node[obj] (D) at (2,0) {};

\node[obj] (E) at (2,1.8) {};

\draw[arrow] (A) -- node[above left, arrowlabel] {$g_{ij}$} (B);

\draw[arrow] (A) -- node[below, arrowlabel] {$g_{ik}$} (C);

\draw[arrow] (B) -- node[above right, arrowlabel] {$g_{jk}$} (C);

\draw[twoarrow] (E) -- node[right, arrowlabel] {$\alpha_{ijk}$} (D);

\end{tikzpicture}
\captionsetup{
    width=5cm,
    singlelinecheck=false
}
\caption{Transition on triple overlap.}
\end{minipage}
\hfill
\begin{minipage}{0.50\textwidth}
\centering
\begin{tikzpicture}[    
node distance=2.0cm and 2.4cm,     
obj/.style={font=\large},     
arrowlabel/.style={font=\small,inner sep=1pt},     
arrow/.style={-{To[length=7pt,width=9pt]},line width=0.6pt} ]

\node[obj] (A) at (0,0) {$\mathcal{C}_i$};

\node[obj] (B) at (3.2,2) {$\mathcal{C}_j$};

\node[obj] (C) at (1.8,-1.5) {$\mathcal{C}_k$};

\node[obj] (D) at (5,0) {$\mathcal{C}_m$};

\draw[arrow] (A) -- node[above left, arrowlabel] {$g_{ij}$} (B);

\draw[arrow] (B) -- node[above left, arrowlabel] {$g_{jk}$} (C);

\draw[arrow] (A) -- node[below left, arrowlabel] {$g_{ik}$} (C);

\draw[arrow] (C) -- node[below right, arrowlabel] {$g_{km}$} (D);

\draw[arrow] (B) -- node[above right, arrowlabel] {$g_{jm}$} (D);

\draw[arrow] (A) -- node[below right, arrowlabel] {$g_{im}$} (D);

\end{tikzpicture}
\captionsetup{
    width=5cm,
    singlelinecheck=false
}
\caption{Tetrahedral diagram on quadruple overlap.}
\end{minipage}
\end{figure}  

The commutative tetrahedral diagram is equivalent to the 2-descent condition on quadruple overlap, which is given by
\begin{equation}
    {\alpha _{ijk}} \circ \left( {{\text{id*}}{\alpha _{jmk}}} \right) = {\alpha _{imk}} \circ \left( {{\alpha _{ijm}}*{\text{id}}} \right),
\end{equation}
for any $(i,j,k) \in I^2$. In this setup, the 1-arrows are braided monoidal 2-equivalences between any pair of representation bicategories, and the 2-arrows are lax natural transformations. 

In fact, we do not need to assume the existence of the 1-arrow and 2-arrow of the 2-descent datum directly. It is more natural to construct the 2-descent conditions at the level of underlying Hopf algebras and twists, rather than at the categorical level. To study the transition of the underlying algebras on double overlaps, we consider that the Hopf-algebra isomorphism $G_{ij}:\left[\mathscr{H}_j\right]\longrightarrow\left[\mathscr{H}_i\right]$ exists for any $(i,j)\in I^2$ due to the physical constraint. We introduce 
\begin{definition}
    Let $\mathscr{H}_i := \mathscr{H}_{U_i}|_{U_{ij}}$ and $\mathscr{H}_j := \mathscr{H}_{U_j}|_{U_{ij}}$ be the quasitriangular quasi-Hopf algebras on the restriction $U_{ij}=U_i\cap U_j$, an isomorphism $G_{ij}$ between the untwisted Hopf algebras $\left[\mathscr{H}_i\right]$ and $\left[\mathscr{H}_j\right]$ is called a \textit{gauge transformation} if it is invertible, satisfies the homomorphism algebra, and preserves the untwisted coproduct (see \cref{fig:gauge trans preserve coproduct}).
\end{definition}
For the quasitriangular quasi-Hopf algebras, there are \textit{twist transformations} $T_{ij}=\mathcal{T}_i \left(G_{ij}\otimes G_{ij}\right)\mathcal{T}_j^{-1}$ that is also invertible because $\mathcal{T}$ and $G_{ij}$ is invertible.

\begin{figure}[H]
\centering
\begin{minipage}{0.35\textwidth}
\centering
\begin{tikzpicture}[    
node distance=2.0cm and 2.4cm,     
obj/.style={font=\large},     
arrowlabel/.style={font=\small,inner sep=1pt},     
arrow/.style={-{To[length=7pt,width=9pt]},line width=0.6pt} ]

\node[obj] (A) at (0,0) {$\left[\mathscr{H}_j\right]$};

\node[obj] (B) at (0,-1.8) {$\left[\mathscr{H}_i\right]$};

\node[obj] (C) at (3,0) {$\left[\mathscr{H}_j\right]\otimes \left[\mathscr{H}_j\right]$};

\node[obj] (D) at (3,-1.8) {$\left[\mathscr{H}_i\right]\otimes \left[\mathscr{H}_i\right]$};

\draw[arrow] (A) -- node[left, arrowlabel] {$G_{ij}$} (B);

\draw[arrow] (A) -- node[above, arrowlabel] {$\Delta$} (C);

\draw[arrow] (B) -- node[below, arrowlabel] {$\Delta$} (D);

\draw[arrow] (C) -- node[right, arrowlabel] {$G_{ij}\otimes G_{ij}$} (D);

\end{tikzpicture}
\captionsetup{
    width=5cm,
    singlelinecheck=false
}
\caption{Gauge transformation between $\left[\mathscr{H}_j\right]$ and $\left[\mathscr{H}_i\right]$ preserves the coproduct.}\label{fig:gauge trans preserve coproduct}
\end{minipage}
\hfill
\begin{minipage}{0.50\textwidth}
\centering
\begin{tikzpicture}[    
node distance=2.0cm and 2.4cm,     
obj/.style={font=\large},     
arrowlabel/.style={font=\small,inner sep=1pt},     
arrow/.style={-{To[length=7pt,width=9pt]},line width=0.6pt} ]

\node[obj] (A) at (0,0) {$\left[\mathscr{H}_j\right]$};

\node[obj] (B) at (2.5,0) {$\left[\mathscr{H}_i\right]$};

\node[obj] (C) at (5,0) {$\left[\mathscr{H}_i\right]\otimes \left[\mathscr{H}_i\right]$};

\node[obj] (D) at (0,-1.8) {$\left[\mathscr{H}_j\right]\otimes\left[\mathscr{H}_j\right]$};

\node[obj] (E) at (5,-1.8) {$\left[\mathscr{H}_i\right]\otimes \left[\mathscr{H}_i\right]$};

\draw[arrow] (A) -- node[above, arrowlabel] {$G_{ij}$} (B);

\draw[arrow] (B) -- node[above, arrowlabel] {$\Delta_i$} (C);

\draw[arrow] (A) -- node[left, arrowlabel] {$\Delta_j$} (D);

\draw[arrow] (D) -- node[below, arrowlabel] {$G_{ij}\otimes G_{ij}$} (E);

\draw[arrow] (E) -- node[right, arrowlabel] {$T_{ij}(\quad)T_{ij}^{-1}$} (C);

\end{tikzpicture}
\captionsetup{
    width=5cm,
    singlelinecheck=false
}
\caption{The transition of twisted coproduct on $U_{ij}$.}\label{fig:transition of twisted coproduct}
\end{minipage}
\end{figure}

\begin{proposition}\label{prop:gauge transformation}
    On $U_{ij}$, suppose that a gauge transformation and a twist transformation exist. The gauge transformation and the twist transformation induce a braided monoidal 2-equivalence, which plays the role of the 1-arrow in the 2-descent datum, between two representation bicategories of quasitriangular quasi-Hopf algebras corresponding to restrictions $U_i \longrightarrow U_{ij}$ and $U_j \longrightarrow U_{ij}$.
\end{proposition}
\begin{proof}
    See Appendix \ref{appendix 2}.
\end{proof}

Proposition \ref{prop:gauge transformation} should be viewed as a bicategorical refinement of the classical Drinfeld twisting Theorem \ref{theorem:twist of quasitri quasi-Hopf}. While it is well known that a Drinfeld twist induces a braided monoidal equivalence between representation categories of quasi-Hopf algebras, we prove that, on double overlaps of admissible coordinate domains, 
the combined gauge and twist transformations naturally define the 1-morphisms of the 2-descent datum for the representation bicategory. Thus, the existence of gauge transformation $G_{ij}$ and twist transformation $T_{ij}$, along with compatible monoidal 2-equivalence $\alpha_{ijk}$, is sufficient for a family of 2-descent data satisfying the 2-descent condition to exist. However, we emphasize that the reverse might not be true. That is, if we initially assume the existence of compatible 2-descent datums $\left(\mathcal{C}_i,g_{ij},\alpha_{ijk}\right)$ at the categorical level, we cannot infer the existence of gauge transformation $G_{ij}$ and twist transformation $T_{ij}$ at the algebraic level.

\subsection{Transition on common refinement of different covering families}
\begin{definition}
    Let $\mathscr{P}:U\longmapsto \text{BRep}(\mathscr{H}_U)$ be a pseudofunctor. For any covering family $\mathcal{U}=\left\{U_i \longrightarrow U\right\}_{i\in I}$, the bicategory $\text{Des}_\mathscr{P}(\mathcal{U})$ of 2-descent datum relative to $\mathcal{U}$ consists of
    \begin{enumerate}
        \item Objects are 2-descent datum $x=\left(\mathcal{C}_i,g_{ij},\alpha_{ijk}\right)$, also called \textit{$\mathscr{P}$-descent datum} (or just descent datum for short),
        \item 1-morphisms that send $x=\left(\mathcal{C}_i,g_{ij},\alpha_{ijk}\right) \longrightarrow x'=\left(\mathcal{C}'_i,g'_{ij},\alpha'_{ijk}\right)$ is defined as
        \begin{equation}
            \begin{array}{rccc}
\psi_i: & \mathcal{C}_i & \longrightarrow & \mathcal{C}'_i \\
 \Lambda_{ij}:  &g'_{ij} \psi_j& \longrightarrow  & \psi_i g'_{ij},
\end{array}
        \end{equation}
        \item 2-morphisms $\Theta_i: \psi_i \longrightarrow \psi'_i$,
    \end{enumerate}
    such that the following diagrams commute.
\end{definition}

\begin{figure}[H]
\centering
\begin{minipage}{0.35\textwidth}
\centering
\begin{tikzpicture}[    
node distance=2.0cm and 2.4cm,     
obj/.style={font=\large},     
arrowlabel/.style={font=\small,inner sep=1pt},     
arrow/.style={-{To[length=7pt,width=9pt]},line width=0.6pt} ]

\node[obj] (A) at (0,0) {$\mathcal{C}_j$};

\node[obj] (B) at (0,-2) {$\mathcal{C}'_j$};

\node[obj] (C) at (2.2,0) {$\mathcal{C}_i$};

\node[obj] (D) at (2.2,-2) {$\mathcal{C}'_i$};

\draw[arrow] (A) -- node[left, arrowlabel] {$\psi_j$} (B);

\draw[arrow] (A) -- node[above, arrowlabel] {$g_{ij}$} (C);

\draw[arrow] (B) -- node[below, arrowlabel] {$g'_{ij}$} (D);

\draw[arrow] (C) -- node[right, arrowlabel] {$\psi_i$} (D);

\end{tikzpicture}
\end{minipage}
\hfill
\begin{minipage}{0.50\textwidth}
\centering
\begin{tikzpicture}[    
node distance=2.0cm and 2.4cm,     
obj/.style={font=\large},     
arrowlabel/.style={font=\small,inner sep=1pt},     
arrow/.style={-{To[length=7pt,width=9pt]},line width=0.6pt},twoarrow/.style={-{Implies[length=7pt,width=9pt]},double,double distance=1.2pt,line width=0.4pt} ]

\node[obj] (A) at (0,0) {$\mathcal{C}_j$};

\node[obj] (B) at (0,-2) {$\mathcal{C}'_j$};

\node[obj] (C) at (3,0) {$\mathcal{C}_i$};

\node[obj] (D) at (3,-2) {$\mathcal{C}'_i$};

\node[obj] (E) at (-0.4,-1) {};

\node[obj] (F) at (0.4,-1) {};

\node[obj] (G) at (2.6,-1) {};

\node[obj] (H) at (3.4,-1) {};

\draw[arrow] (A) to[out=-150,in=150] node[left,arrowlabel,pos=0.48] {$\psi_j$} (B);

\draw[arrow] (A) to[out=-30,in=30] node[right,arrowlabel,pos=0.48] {$\psi'_j$} (B);

\draw[arrow] (C) to[out=-150,in=150] node[left,arrowlabel,pos=0.48] {$\psi_i$} (D);

\draw[arrow] (C) to[out=-30,in=30] node[right,arrowlabel,pos=0.48] {$\psi'_i$} (D);

\draw[arrow] (A) -- node[above, arrowlabel] {$g_{ij}$} (C);

\draw[arrow] (B) -- node[below, arrowlabel] {$g'_{ij}$} (D);

\draw[twoarrow] (E) -- node[below, arrowlabel] {$\Theta_j$} (F);

\draw[twoarrow] (G) -- node[below, arrowlabel] {$\Theta_i$} (H);

\end{tikzpicture}
\end{minipage}
\end{figure}
\begin{definition}\label{def:biequivalence}
    Let $x$ and $y$ be objects in bicategory $\text{Des}_\mathscr{P}(\mathcal{U})$. A 1-morphism $P:x \longrightarrow y$ is called an object biequivalence between $x$ and $y$ if there exists a 1-morphism $Q:y\longrightarrow x$ such that $QP \cong \text{id}_x$ and $PQ \cong \text{id}_y$.
\end{definition}

For any covering family $\mathcal{U}=\left\{U_i \longrightarrow U\right\}_{i\in I}$ of $U$, we define the refinement of $\mathcal{U}$ is a covering family $\mathcal{V}=\left\{V_a \longrightarrow U\right\}_{a\in A}$ if and only if there is a $U_i \in \mathcal{U}$ such that $V_a \subseteq  U_i$ for any $a \in A$. The refinement of $\mathcal{U}$ by $\mathcal{V}$ is denoted by an arrow $r:\mathcal{V} \longrightarrow \mathcal{U},a\longmapsto i(a)$. A refinement $r:\mathcal{V} \longrightarrow \mathcal{U}$ induces a 2-functor 
\begin{equation}
   \begin{array}{rccc}
r^*: & \text{Des}_\mathscr{P}(\mathcal{U})& \longrightarrow  & \text{Des}_\mathscr{P}(\mathcal{V}), \\
 &\left(\mathcal{C}_i,g_{ij},\alpha_{ijk}\right) & \longmapsto  & \left({\mathcal{C}_{i(a)}}|_{V_a},g_{i(a)j(b)}|_{V_{ab}},\alpha_{i(a)j(b)k(c)}|_{V_{abc}}\right),
 \\
 & \left(\psi_i,\Lambda_{ij}\right) &\longmapsto & \left(\psi_{i(a)}|_{V_a},\Lambda_{i(a)j(b)}|_{V_{ab}}\right), \\
 & \Theta_i &\longmapsto & \Theta_{i(a)}|_{V_a}.
\end{array}
\end{equation}
For any two covers $\mathcal{U}=\left\{U_i \longrightarrow U\right\}_{i\in I}$ and $\mathcal{V}=\left\{V_j \longrightarrow U\right\}_{j\in J}$, the common refinement of $\mathcal{U}$ and $\mathcal{V}$ is given by
\begin{equation}
    \mathcal{U} \wedge \mathcal{V} = \left\{U_i \cap V_j \longrightarrow U\right\}_{i\in I,j\in J},\qquad U_i \cap V_j\neq \varnothing \quad \forall i\in I,j\in J.
\end{equation}
The refinement maps $r: \mathcal{U} \wedge \mathcal{V}\longrightarrow \mathcal{U}$ and $s: \mathcal{U} \wedge \mathcal{V} \longrightarrow \mathcal{V}$ allow us to compare the descent datum that lies in different covering families. By taking the pullbacks of object $x$ in $\text{Des}_\mathscr{P}(\mathcal{U})$ and $y$ in $\text{Des}_\mathscr{P}(\mathcal{V})$, we have the descent datum $r^*x$ and $s^* y$ in $\text{Des}_\mathscr{P}( \mathcal{U} \wedge \mathcal{V} )$, which can be compared, for example, using Definition \ref{def:biequivalence}.

Next, we define a class of descent representations over an arbitrary open coordinate domain $U$, including $U$ that is nonadmissible. Let $\text{Obj}\left(\text{Des}_\mathscr{P}(\mathcal{U})\right)$ be the collection of objects in bicategory $\text{Des}_\mathscr{P}(\mathcal{U})$ relative to covering family $\mathcal{U}$, a class of descent representations over $U$ is given by
\begin{equation}
    \text{DRep}_\mathscr{P}(U) := \coprod\limits_{\mathcal{U}\in\text{ACov}(U)} {\text{Obj}\left(\text{Des}_\mathscr{P}(\mathcal{U})\right)} ,
\end{equation}
where $\text{ACov}(U)$ is the category of admissible covers of $U$ that contains
\begin{equation}
    \text{Obj}\left(\text{ACov}(U)\right) = \left\{\mathcal{U}=\left\{U_i\longrightarrow U\right\}_{i\in I}\bigg|U_i \in \text{Loc}\mathcal{M},\bigcup\limits_{i\in I}U_i = U\right\},
\end{equation}
and morphisms are refinements. Thus, $\text{DRep}_\mathscr{P}(U)$ does not depend on the choice of covering $\mathcal{U}$. The elements of $\text{DRep}_\mathscr{P}(U)$, called \textit{descent representations}, are pairs of a covering family $\mathcal{U}$ and an object in the bicategory of descent datum relative to it.

\begin{theorem}\label{theorem:modulo equivalence relation}
    Suppose there exists a common refinement $\mathcal{U} \wedge \mathcal{V}$ of $\mathcal{U}$ and $\mathcal{V}$ with refinement maps $r:\mathcal{U} \wedge \mathcal{V} \longrightarrow \mathcal{U}$ and $s:\mathcal{U} \wedge \mathcal{V} \longrightarrow \mathcal{V}$. The biequivalence in $\text{Des}_\mathscr{P}(\mathcal{U} \wedge \mathcal{V})$ defines an equivalence relation on $\text{DRep}_\mathscr{P}(U)$.
\end{theorem}
\begin{proof}
    Let $(\mathcal{U},x)$ and $(\mathcal{V},y)$ be two descent representations in $\text{DRep}_\mathscr{P}(U)$. The corresponding pullback descent datum is given by objects $r^*x$ and $s^*y$ in $\text{Des}_\mathscr{P}(\mathcal{U} \wedge \mathcal{V})$ such that $P:r^*x\longrightarrow s^*y$ is a biequivalence. We have the equivalence relation $(\mathcal{U},x) \sim_{\text{ref,bieq}} (\mathcal{V},y)$ on $\text{DRep}_\mathscr{P}(U)$.
    \begin{enumerate}
        \item Reflexivity: Taking the common refinement to be $\mathcal{U}$ itself, $\text{id}_\mathcal{U}:\mathcal{U} \longrightarrow \mathcal{U}$. The biequivalence is given by $\text{id}^*_\mathcal{U}: x \longrightarrow x$. Hence, $(\mathcal{U},x) \sim_{\text{ref,bieq}} (\mathcal{U},x)$.
        \item Symmetry: Assume $(\mathcal{U},x) \sim_{\text{ref,bieq}} (\mathcal{V},y)$ with a biequivalence $P:r^* x \longrightarrow s^*y$. Thus, there exists a biequivalence $Q:s^*y \longrightarrow r^*x$ such that $QP \cong \text{id}_x$ and $PQ \cong \text{id}_y$, which yields $(\mathcal{V},y)\sim_{\text{ref,bieq}} (\mathcal{U},x)$.
        \item Transitivity: Assume $(\mathcal{U},x) \sim_{\text{ref,bieq}} (\mathcal{V},y)$ and $(\mathcal{V},y) \sim_{\text{ref,bieq}} (\mathcal{W},z)$ refer to common refinements $\mathcal{U} \wedge \mathcal{V}$ with refinement maps $r,s:\mathcal{U} \wedge \mathcal{V}\longrightarrow \mathcal{U},\mathcal{V}$ and $\mathcal{V}\wedge \mathcal{W}$ with refinement maps $p,q:\mathcal{V} \wedge \mathcal{W}\longrightarrow \mathcal{V},\mathcal{W}$, respectively. Biequivalence on $\text{Des}_\mathscr{P}(\mathcal{U} \wedge \mathcal{V})$ is $F$ while biequivalence on $\text{Des}_\mathscr{P}(\mathcal{V} \wedge \mathcal{W})$ is $G$. Choose the common refinement $\mathcal{Z}=(\mathcal{U} \wedge \mathcal{V})\wedge (\mathcal{V} \wedge \mathcal{W})$ with refinement maps $u:\mathcal{Z} \longrightarrow \mathcal{U} \wedge \mathcal{V}$ and $t:\mathcal{Z} \longrightarrow \mathcal{V} \wedge \mathcal{W}$, we have
        \begin{equation}
            u^* r^* x \mathop{\longrightarrow}^{u^*F} u^*s^*y, \quad \text{and}\quad  t^* p^*y\mathop{\longrightarrow}^{t^*G} t^*q^*z.
        \end{equation}
        Both $u^*s^*y$ and $t^* p^*y$ are the pullbacks of the same descent datum $y$ into the common refinement $\mathcal{Z}$, thus, there is a biequivalence $u^*s^*y \cong t^* p^*y$ in $\text{Des}_\mathscr{P}(\mathcal{Z})$, which implies the equivalence relation $(\mathcal{U},x) \sim_{\text{ref,bieq}} (\mathcal{W},z)$. 
    \end{enumerate}
    Therefore, $\sim_{\text{ref,bieq}}$ is an equivalence relation.
\end{proof}
Theorem \ref{theorem:modulo equivalence relation} allows us to classify descent data that consists of different covering families by comparing the descent data over common refinements. A refinement gives a more detailed atlas, not a new geometry. Therefore, two presentations represent the same stacky object precisely when they agree, up to equivalence, after both are pulled back to a common refinement. 

\subsection{Stacky description of $R$-space}
\begin{definition}
    A bicategory is called a \textit{bicategory-valued stack} (or \textit{weak 2-stack} over site $\mathscr{S}$ if it is a bicategory-valued prestack (weak 2-prestack) fibered over site $\mathscr{S}$ such that every 2-descent datum is effective and the 2-functor which associates to an object its canonical descent datum is an equivalence.
\end{definition}
However, there are some problems we might have to face. Firstly, requiring every descent datum to be effective is not natural. Secondly, one requires that all the descent data are isomorphic to the restriction of some global objects, which may include the non-admissible coordinate domains. Thus, we need to reconstruct such global descriptions of $R$-space that associate to non-admissible coordinate domains by gluing the local descent data instead of assigning any non-admissible coordinate domain to a representation bicategory. That is why we have to define the class of descent representations over an arbitrary open coordinate domain $U$ and the equivalence relation on it (see Theorem \ref{theorem:modulo equivalence relation}). 
\begin{definition}
    The 2-stackification of a bicategory-valued prestack $\mathscr{P}:U_i\mapsto \text{BRep}(\mathscr{H}_{U_i})$ over a arbitrary open coordinate domain $U$, is given by
    \begin{equation}
        \mathscr{P}^+(U) := \left( \coprod\limits_{\mathcal{U}\in\text{ACov}(U)} {\text{Des}_\mathscr{P}(\mathcal{U})} \right)\bigg/\sim_{\text{ref,bieq}} = \mathop{2\text{-colim}}\limits_{\mathcal{U}\in\text{ACov}(U)}{\text{Des}_\mathscr{P}(\mathcal{U})}.
    \end{equation}
\end{definition}
\begin{proposition}
    Let $\mathscr{S}_{\text{all}}=(\text{Open}\mathcal{M},J)$ be the site, where $\text{Open}\mathcal{M} \supset \text{Loc}\mathcal{M}$ is the category of all open coordinate domains $U \subset \mathcal{M}$. Then,
    \begin{equation}
        \begin{array}{rccl}
\mathscr{P}^+: & \mathscr{S}_{\text{all}}^{\text{op}}& \longrightarrow & \textbf{BCatRe}, \\
 & U & \longmapsto  & \left\{\begin{array}{cl}
      \text{BRep}(\mathscr{H}_U)&\text{if } U \in \text{Loc}\mathcal{M},  \\
      \mathop{2\text{-colim}}\limits_{\mathcal{U}\in\text{ACov}(U)}{\text{Des}_\mathscr{P}(\mathcal{U})}&\text{if } U \in \text{Open}\mathcal{M}\backslash\text{Loc}\mathcal{M}
 \end{array}\right.,
\end{array}
    \end{equation}
    is a pseudofunctor.
\end{proposition}
\begin{proof}
    If $U \in \text{Loc}\mathcal{M}$, we have a natural embedding 
    \begin{equation}
        \iota_U: \mathscr{P}(U)  \longrightarrow \mathscr{P}^+(U).
    \end{equation}
    Therefore, for any $U \in \text{Loc}\mathcal{M}$, $U \longmapsto \mathscr{P}^+(U)=\mathscr{P}(U)$ is a pseudofunctor, which is compatible with the inclusion $\iota_{UV}:V\xhookrightarrow{} U$ by 
    \begin{equation}
        \text{BRep}(\mathscr{H}_U)=\mathscr{P}^+(U)\mathop{\longrightarrow}\limits^{\iota^*_{UV}} \mathscr{P}^+(U|_V)= \text{BRep}(\mathscr{H}_{U}|_V)\mathop{\longrightarrow}\limits^{\cong} \text{BRep}(\mathscr{H}_{V})= \mathscr{P}^+(V).
    \end{equation}
    for any $V\subset U \in \text{Loc}\mathcal{M}$.
If $U \in \text{Open}\mathcal{M}\backslash\text{Loc}\mathcal{M}$, an object of $\mathscr{P}^+(U)$ is represented by a descent datum over some admissible cover $\mathcal{U}=\left\{U_i\longrightarrow U\right\}_{i \in I}$. By pulling this cover back to $V$, and refining if necessary, one obtains an admissible coordinate cover $\mathcal{U}|_V=\left\{U_i \cap V \longrightarrow V\right\}$ because $\bigcup\limits_{i\in I}U_i \cap V = U \cap V=V$. Hence, $\mathcal{U}|_V \in \text{ACov}(V)$ and
    \begin{equation}
        \mathscr{P}^+(U)\mathop{\longrightarrow}\limits^{\iota^*_{UV}}\mathscr{P}^+(U|_V)=\mathop{2\text{-colim}}\limits_{\mathcal{U}|_V\in\text{ACov}(V)}{\text{Des}_\mathscr{P}(\mathcal{U}|_V)} \mathop{\longrightarrow}\limits^{\cong} \mathscr{P}^+(V).
    \end{equation}
We also have identity $\text{id}_U:U\xhookrightarrow{} U$ corresponding to canonical equivalence $\text{id}_U^* \cong \text{id}_{\mathscr{P}^+(U)}$ and the composition $W \mathop{\xhookrightarrow{}}\limits^{\iota_{UV}}V \mathop{\xhookrightarrow{}}\limits^{\iota_{VW}} U$ corresponding to $\iota^*_{VW}\iota^*_{UV} \cong (\iota_{UV}\circ \iota_{VW})^*$. Therefore, $\mathscr{P}^+$ is a pseudofunctor.
\end{proof}

\begin{theorem}\label{theorem:stackification}
    For any cover $\mathcal{U}=\left\{U_i\longrightarrow U\right\}_{i\in I}$, there exists a bicategory equivalence $\text{res}_\mathcal{U}:\mathscr{P}^+(U) \mathop{\longrightarrow}\limits^{\cong} \text{Des}_{\mathscr{P}^+}(\mathcal{U})$. Therefore, the pseudofunctor $\mathscr{P}^+:\mathscr{S}_{\text{all}}^{\text{op}}\longrightarrow \textbf{BCatRe}$ defines a bicategory-valued stack (weak 2-stack).
\end{theorem}
\begin{proof}
    See Appendix \ref{appendix 3}.
\end{proof} 

In conclusion, the stacky description developed above shows that the $R$-space should not be regarded as a conventional global geometry obtained by gluing coordinate domains through descent theory. Rather, its natural global object is a weak 2-stack whose local fibers are braided monoidal representation categories associated with quasitriangular quasi-Hopf algebras. In this sense, the noncommutative, nonassociative, and locally nongeometric features of the $R$-flux background are encoded into the categorical structure over coordinate domains of the underlying manifold $\mathcal{M}$. By gluing the local data within the framework of descent theory, our results provide a global and geometric-like definition of $R$-space without requiring the existence of a global algebra. The $R$-space over manifold $\mathcal{M}$ is given by
\begin{equation}
        \begin{array}{rccl}
\mathscr{P}^+: & \mathscr{S}_{\text{all}}^{\text{op}}& \longrightarrow & \textbf{BCatRe}, \\
 & U & \longmapsto  & \left\{\begin{array}{cl}
      \text{BRep}(\mathscr{H}_U)&\text{if } U \in \text{Loc}\mathcal{M},  \\
      \mathop{2\text{-colim}}\limits_{\mathcal{U}\in\text{ACov}(U)}{\text{Des}_\mathscr{P}(\mathcal{U})}&\text{if } U \in \text{Open}\mathcal{M}\backslash\text{Loc}\mathcal{M}
 \end{array}\right..
\end{array}
\end{equation}
Let $\mathcal{V}=\left\{V_i\longrightarrow U\right\}_{i\in I}$ be an admissible cover of $U$ (if $\mathcal{V}$ is not admissible, we continue to choose a family of covers for each $V_i$ until we obtain an admissible covering family). The restriction of $\mathop{2\text{-colim}}\limits_{\mathcal{U}\in\text{ACov}(U)}{\text{Des}_\mathscr{P}(\mathcal{U})}$ on $V_i$ is given by an equivalence class $\mathcal{A}=\left[\mathcal{W}, \mathcal{C}_{ia},g_{iab},\alpha_{iabc} \right]_{\text{ref},2\text{eq}}\in \mathscr{P}^+(V_i)$, where $\mathcal{W}$ is an admissible covering of $V_i$. A natural choice of $\mathcal{W}$ is $\{V_i \longrightarrow V_i\}$. Thus, this descent datum is biequivalent to a $\mathscr{P}$-descent datum relative to $\{V_i \longrightarrow V_i\}$, which encodes exactly the representation bicategory of a quasitriangular quasi-Hopf algebra, up to a biequivalence, on an admissible coordinate domain,
\begin{equation}
    \mathop{2\text{-colim}}\limits_{\mathcal{U}\in\text{ACov}(U)}{\text{Des}_\mathscr{P}(\mathcal{U})}\Big|_{V_i} \cong \left(\text{BRep}(\mathscr{H}_i),g_{ij},\alpha_{ijk}\right),
\end{equation}
without losing any information. For a $R$-space over an admissible coordinate domain $V\subset\mathcal{M}$, the $R$-space is simply a representation category of a quasitriangular quasi-Hopf algebra fibered over $V$. For a non-admissible open subset $U$, we cannot define a local $R$-space on $U$ by a single quasitriangular quasi-Hopf algebra $\mathscr{H}_U$. Instead, we define the $R$-space on $U$ by the quotient of the disjoint union of all categories of descent datum relative to an admissible covering family $\mathcal{U}$ of $U$ over a biequivalence relation. 

%%%%%%%%%%%%%%%%%%%%%%%%%%%%%%%%%%%%%%%%%%%%%%%%%%%%%%%%%%%%%%%%%%%%%%%%%%%%%%%%%%%
\section{Existence of admissible coordinate domains}\label{sec:Admissible coordinate domains}
In the previous section, we systematically derived a geometric and global definition for nongeometric $R$-space by using descent theory on the site of admissible coordinate domains. One may ask how to choose admissible coordinate domains, equivalently, how to choose an admissible covering family for any open set $U$. Suppose we can choose a collection of open subsets of $\mathcal{M}$, how to estimate whether an arbitrary subset is admissible or not? Technically, there is no way to do that if $\mathcal{M}$ is non-Hausdorff. However, if $\mathcal{M}$ is Hausdorff and metrizable, it admits a distance function $d:\mathcal{M}\times \mathcal{M} \longrightarrow \mathbb{R}$. The idea is to compare the distance between any two points with a value derived from physical constraints. In the physical aspect, all the phenomena of $R$-space, or string theory, are only observable at some scale. For example, we cannot experience in daily life the effects of quantum gravity, which are located at an extremely high energy scale, equivalently, an extraordinarily small length. Therefore, all the phenomena of $R$-space should be located at the scale that the noncommutative relation
\begin{equation}
    [x^\mu,x^\nu] = i\hbar \theta^{\mu\nu}
\end{equation}
works well. We note that $\hbar$ is just some formal constant to define the $\hbar$-adic topology of deformation quantization, which represents \textit{quantum} and leads to trivial commutative and associative theories at the classical limit $\hbar \to 0$. Suppose denoting $L$ to be the scale of length where the effects are observable. The physical scale length $L$ can adopt some candidates, for example, Planck length $l_P = \sqrt{\hbar G/c^3} \sim 10^{-28}\ \text{eV}^{-1}$, supersymmetry scale $l_s \sim10^{-25} - 10^{-27}\ \text{eV}^{-1}$, or just some undiscovered scale $\sqrt{\hbar \theta}$ refers to the noncommutative parameter $\theta$. 

To estimate the admissibility of any open domain $U$, we restrict our attention to coordinate domains that are sufficiently small relative to the characteristic length scale $L$. A possible definition of admissible open domains $U$ is given by
\begin{definition}\label{def:admissible}
    Let $L$ be the physical scale length, $\mathcal{M}$ be a Hausdorff and metrizable manifold. An open coordinate domain $U \subset \mathcal{M}$ is admissible if there exists $p \in \mathcal{M}$ such that $U  \subseteq B_{L/2}(p)$, where $B_{L/2}(p)=\left\{q\in \mathcal{M}:d(p,q)<L/2\right\}$ is an open ball centered at point $p$ with radius $L/2$.
\end{definition}
Any admissible coordinate domain $U$ can be described by a single quasitriangular quasi-Hopf algebra $\mathscr{H}_U$.

\begin{lemma}
Every open subset $O\subset \mathcal M$ admits an admissible covering family.
\end{lemma}

\begin{proof}
Let $O\subset \mathcal M$ be an open subset. For each point $x\in O$, set $W_x:=O\cap B_{L/2}(x)$. Since $O$ is open and $B_{L/2}(x)$ is open in the topology induced by $d$, the set $W_x$ is an open neighborhood of $x$. Since $\mathcal M$ is a manifold, there exists a coordinate chart $(V_x,\varphi_x)$ with $x\in V_x$. Define
\begin{equation}
    U_x:=V_x\cap W_x=V_x\cap O\cap B_{L/2}(x).
\end{equation}
Then $U_x$ is open, $x\in U_x$, $U_x\subset O$, and $U_x\subset B_{L/2}(x)$. Moreover, $U_x$ is a coordinate domain, since the restriction $\varphi_x|_{U_x}:U_x\to \varphi_x(U_x)$ is again a coordinate chart. Thus $U_x$ is admissible by Definition \ref{def:admissible}. For any $y\in \bigcup\limits_{x\in O} U_x$, there exists $x_0 \in O$ such that $y\in U_{x_0}=V_{x_0}\cap O\cap B_{L/2}(x_0)$, then we have $y\in O$. Thus, $\bigcup\limits_{x\in O} U_x \subseteq  O$. On the other hand, for any $z\in O$, there exists an admissible coordinate domain $U_z$ in the family $\{U_x\}_{x\in O}$. Furthermore, for any $z \in O$, we have $z \in V_z$ because $V_z$ is a coordinate neighborhood of $z$. We also have $z \in B_{L/2}(z)$ because $d(z,z)=0<L/2$. Hence, $z \in V_z \cap O\cap B_{L/2}(z)=U_z$ for any $z\in O$, which implies $O \subseteq \bigcup\limits_{x\in O} U_x$. Therefore, $O=\bigcup\limits_{x\in O} U_x$, and $\left\{U_x \longrightarrow O\right\}_{x\in O}$ is an admissible covering family of open set $O$.

\end{proof}

%%%%%%%%%%%%%%%%%%%%%%%%%%%%%%%%%%%%%%%%%%%%%%%%%%%%%%%%%%%%%%%%%%%%%%%%%%%%%%%%%%%
\appendix
\section{Proof of Theorem \ref{theorem:twist of quasitri quasi-Hopf}}\label{appendix 1}
Firstly, the proof is similar to Theorem \ref{theorem for new quasitriangular Hopf alg}, but instead of using the definition for Hopf algebra, we use the definition for quasi-Hopf algebra. Secondly, it is important to show that $\phi_\mathcal{T}$ is a 3-cocycle. First, we define the coface maps $\text{d}_i=\textrm{id}_{\mathscr{H}^{\otimes (i-1)}}\otimes \Delta \otimes\textrm{id}_{\mathscr{H}^{\otimes (n-i)}}:\mathscr{H}^{\otimes n} \longrightarrow \mathscr{H}^{\otimes (n+1)}$ for $1\leq i \leq n$, $\text{d}_0=1\otimes\textrm{id}_{\mathscr{H}^{\otimes n}}$, and $\text{d}_{n+1}=\textrm{id}_{\mathscr{H}^{\otimes n}} \otimes 1$. The Majid operators are defined as
    \begin{equation}
        \begin{split}
            \text{d}_+ \mathcal{F}=\mathop \prod \limits_{i{\textrm{ even}}} {\text{d}_i}\mathcal{F}, \qquad \qquad \text{d}_- \mathcal{F}=\mathop \prod \limits_{i{\textrm{ odd}}} {\text{d}_i}\mathcal{F}, \qquad \qquad \forall\mathcal{F} \in \mathscr{H}^{\otimes n},
        \end{split}
    \end{equation}
    If $\mathcal{F}$ is a $n$-cochain, then its coboundary is
    \begin{equation}
        \text{d}\mathcal{F}=(\text{d}_+\mathcal{F})(\text{d}_-\mathcal{F})\qquad \qquad \forall\mathcal{F} \in \mathscr{H}^{\otimes n},
    \end{equation}
    Applying for 3-cocycle $\phi \in \mathscr{H}\otimes \mathscr{H} \otimes \mathscr{H}$, the 3-cocycle condition reads
    \begin{equation}
        (\text{d}_0\phi) (\text{d}_2  \phi)(\text{d}_4  \phi) = (\text{d}_3  \phi)(\text{d}_1  \phi).
    \end{equation}
    Now we define the twisted 3-cocycle $\phi_\mathcal{T} \in \mathscr{H}_\mathcal{T}\otimes \mathscr{H}_\mathcal{T}\otimes \mathscr{H}_\mathcal{T}$ as
    \begin{equation}
        \phi_\mathcal{T}= (\text{d}_+ \mathcal{T})\phi (\text{d}_-\mathcal{T}^{-1}).
    \end{equation}
    The 3-cocycle condition for $\phi_\mathcal{T}$ is given by $(\text{d}_0^{\mathcal{T}}\phi) (\text{d}_2^{\mathcal{T}}  \phi)(\text{d}_4^{\mathcal{T}}  \phi) = (\text{d}_3^{\mathcal{T}}  \phi)(\text{d}_1^{\mathcal{T}}  \phi)$, where $d^\mathcal{T}_i =\textrm{id}_{\mathscr{H}^{\otimes (i-1)}}\otimes \Delta_\mathcal{T} \otimes\textrm{id}_{\mathscr{H}^{\otimes (n-i)}}:\mathscr{H}_\mathcal{T}^{\otimes n} \longrightarrow \mathscr{H}_\mathcal{T}^{\otimes (n+1)} $ (replace $\Delta$ by $\Delta_\mathcal{T}$). We will show $\phi_\mathcal{T}$ satisfies this condition. We have
    \begin{equation}
        \begin{split}
            \left(\text{d}^\mathcal{T}_0 \phi_\mathcal{T}\right)\left(\text{d}^\mathcal{T}_2 \phi_\mathcal{T}\right)\left(\text{d}^\mathcal{T}_4 \phi_\mathcal{T}\right)&=\left[\left(\text{d}_0\text{d}_0 \mathcal{T}\right)\left(\text{d}_0\text{d}_2 \mathcal{T}\right)\left(\text{d}_0 \phi\right)\left(\text{d}_0\text{d}_1 \mathcal{T}^{-1}\right)\left(\text{d}_0\text{d}_3 \mathcal{T}^{-1}\right)\right]\\
            &\qquad \left[\mathcal{T}_{23}\left(\text{d}_2\text{d}_0 \mathcal{T}\right)\left(\text{d}_2\text{d}_2 \mathcal{T}\right)\left(\text{d}_2 \phi\right)\left(\text{d}_2\text{d}_1 \mathcal{T}^{-1}\right)\left(\text{d}_2\text{d}_3 \mathcal{T}^{-1}\right)\mathcal{T}_{23}^{-1}\right]\\
            &\qquad\qquad \left[ \left(\text{d}_4\text{d}_0 \mathcal{T}\right)\left(\text{d}_4\text{d}_2 \mathcal{T}\right)\left(\text{d}_4\phi\right)\left(\text{d}_2\text{d}_3 \mathcal{T}^{-1}\right)\left(\text{d}_4\text{d}_1 \mathcal{T}^{-1}\right)\left(\text{d}_4\text{d}_3 \mathcal{T}^{-1}\right) \right]\\
            &=\left(\text{d}_0\text{d}_0 \mathcal{T}\right)\left(\text{d}_0\text{d}_2 \mathcal{T}\right)\left(\text{d}_0 \phi\right)\left(\text{d}_2\text{d}_2 \mathcal{T}\right)\left(\text{d}_2 \phi\right)\left(\text{d}_2\text{d}_1 \mathcal{T}^{-1}\right)
            \left(\text{d}_4\phi\right)\\
            &\qquad\qquad\left(\text{d}_2\text{d}_3 \mathcal{T}^{-1}\right)\left(\text{d}_4\text{d}_1 \mathcal{T}^{-1}\right)\left(\text{d}_4\text{d}_3 \mathcal{T}^{-1}\right)\\
            &=\left(\text{d}_0\text{d}_0 \mathcal{T}\right)\left(\text{d}_0\text{d}_2 \mathcal{T}\right)\left(\text{d}_3\text{d}_2 \mathcal{T}\right)\left(\text{d}_0\phi\right)\left(\text{d}_2\phi\right)\left(\text{d}_4\phi\right)\\
            &\qquad\qquad\left(\text{d}_1\text{d}_1 \mathcal{T}^{-1}\right)\left(\text{d}_4\text{d}_1 \mathcal{T}^{-1}\right)\left(\text{d}_4\text{d}_3 \mathcal{T}^{-1}\right)\\
            &=\mathcal{T}_{34}\left(\text{d}_3\text{d}_0 \mathcal{T}\right)\left(\text{d}_3\text{d}_2 \mathcal{T}\right)\left(\text{d}_3\phi\right)\left(\text{d}_3\text{d}_1 \mathcal{T}^{-1}\right)\left(\text{d}_3\text{d}_3 \mathcal{T}^{-1}\right)\mathcal{T}_{34}^{-1}\\
            &\qquad \mathcal{T}_{12}\left(\text{d}_1\text{d}_0 \mathcal{T}\right)\left(\text{d}_1\text{d}_2 \mathcal{T}\right)\left(\text{d}_1\phi\right)\left(\text{d}_1\text{d}_1 \mathcal{T}^{-1}\right)\left(\text{d}_1\text{d}_3 \mathcal{T}^{-1}\right)\mathcal{T}_{12}^{-1}\\
            &=\left(\text{d}_3^\mathcal{T}\phi_\mathcal{T}\right)\left(\text{d}_1^\mathcal{T}\phi_\mathcal{T}\right).
        \end{split}
    \end{equation}
    Moreover, the 3-cocycle $\phi_\mathcal{T}$ can also be written in the form of
    \begin{equation}
        \phi_\mathcal{T} = \left(1\otimes \mathcal{T}\right)\left((\textrm{id}\otimes\Delta) \mathcal{T}\right)\phi \left((\Delta \otimes \textrm{id})\mathcal{T}^{-1}\right)\left(\mathcal{T}^{-1}\otimes 1\right),
    \end{equation}
    due to the relation $\Delta_\mathcal{T}(a)=\mathcal{T}\Delta(a)\mathcal{T}^{-1}$. Thus, $\phi_\mathcal{T}$ is an invertible 3-cocyle and counital if $\phi$ and $\mathcal{T}$ are.

\section{Proof of Proposition \ref{prop:gauge transformation}}\label{appendix 2}
We can distinguish the underlying Hopf algebra and the cochain twist in a quasitriangular quasi-Hopf algebra as $\mathscr{H}_{i}=\left(\left[\mathscr{H}_i\right],\mathcal{T}_i\right)$. The 1-arrow $g_{ij}=\left(G_{ij},T_{ij}\right)$ satisfies $\varepsilon_i\circ g_{ij}=\varepsilon_j$, and the transition on any double overlap $U_i \cap U_j$ is given by
\begin{equation}
     \begin{array}{rccc}
G_{ij}: & \left[\mathscr{H}_j\right]& \mathop{\longrightarrow}\limits^{\cong}  & \left[\mathscr{H}_i\right], \\
T_{ij} :  & \mathcal{T}_j & \longrightarrow  & \mathcal{T}_i,
\end{array}
\end{equation}
where $G_{ij}$ is invertible and satisfies the homomorphism algebra, $G_{ij}(1_j)=1_i$, $G_{ij}(ab)=G_{ij}(a)G_{ij}(b)$ for all $a,b\in \left[\mathscr{H}_j\right]$. The twist transformation $T_{ij}=\mathcal{T}_i \left(G_{ij}\otimes G_{ij}\right)\mathcal{T}_j^{-1}$ is also invertible because $\mathcal{T}$ and $G_{ij}$ is invertible. The gauge transformation $G_{ij}$ preserves the untwisted coproduct $\Delta$ (see \cref{fig:gauge trans preserve coproduct}), which implies
\begin{equation}\label{eq:gauge trans preserve coproduct}
    \left(G_{ij}\otimes G_{ij}\right)\Delta(h) = \Delta\left(G_{ij}(h)\right), \qquad \qquad \forall h \in \left[\mathscr{H}_j\right].
\end{equation}

To see how the transition between two quasitriangular quasi-Hopf algebras on a double overlap induces the braided monoidal 2-equivalence, we need to determine the action of $g_{ij}$ on the twisted structures, such as the twisted coproduct, twisted $\mathcal{R}$-matrix, and twisted 3-coboundary. We denote the twisted coproduct, twisted $\mathcal{R}$-matrix, and the twisted 3-coboundary of the quasitriangular quasi-Hopf algebra $\mathscr{H}_{U_i}$ as $\Delta_i$, $\mathcal{R}_i$, and $\phi_i$, respectively. First, for any $h\in \left[\mathscr{H}_i\right]$, we introduce the image of $\Delta_j$ under the gauge transformation, which is given by
\begin{equation}
    \Delta_j^{G_{ij}}(h) = \left(G_{ij}\otimes G_{ij}\right)\Delta_j\left(G_{ij}^{-1}(h)\right) \in \left[\mathscr{H}_i\right] \otimes \left[\mathscr{H}_i\right].
\end{equation}
It is important to notice that $\Delta_j^{G_{ij}}$ does not coincide with $\Delta_i$, in fact, they are different up to an isomorphism. By using identity (\ref{eq:gauge trans preserve coproduct}) and the definition of $T_{ij}$, we have
\begin{equation}
    \begin{split}
        \Delta_j^{G_{ij}}(h)&= \left(G_{ij}\otimes G_{ij}\right)\left(\mathcal{T}_j\Delta\left(G_{ij}^{-1}(h)\right) \mathcal{T}_j^{-1}\right) \\
        &= \left(G_{ij}\otimes G_{ij}\right)(\mathcal{T}_j)\Delta(h)  \left(G_{ij}\otimes G_{ij}\right)(\mathcal{T}_j^{-1})\\
        &=T_{ij}^{-1}\mathcal{T}_i \Delta(h) \mathcal{T}_i^{-1} T_{ij}\\
        &=T_{ij}^{-1} \Delta_i(h)  T_{ij},
    \end{split}
\end{equation}
which implies $\Delta_i(h) = T_{ij}\Delta_j^{G_{ij}}(h)T_{ij}^{-1}$ and the diagram in \cref{fig:transition of twisted coproduct} commute. Thus, one can perform the transition of a twisted operator on a double overlap by applying the gauge transformation, then replacing the twist by the twist transformation. For example, one can compute the 3-coboundary of $\mathscr{H}_{U_i}$
\begin{equation}
     \phi_i =\text{d}\mathcal{T}_i= \left(1\otimes T_{ij}\right)\left[\left(\textrm{id}\otimes\Delta_j^{G_{ij}}\right) T_{ij}\right]\phi_j^{G_{ij}} \left[\left(\Delta_j^{G_{ij}} \otimes \textrm{id}\right)T_{ij}^{-1}\right]\left(T_{ij}^{-1}\otimes 1\right),
\end{equation}
where $ \phi_j^{G_{ij}} = \left(G_{ij}\otimes G_{ij}\otimes G_{ij}\right)\phi_j\in \left[\mathscr{H}_i\right]\otimes \left[\mathscr{H}_i\right]\otimes\left[\mathscr{H}_i\right]$. One can also compute the $\mathcal{R}$-matrix of $\mathscr{H}_{U_i}$, which is given by
\begin{equation}
    \mathcal{R}_i = T_{ij,21}\mathcal{R}_j^{G_{ij}}T_{ij}^{-1}= T_{ij,21}\left[\left( G_{ij}\otimes G_{ij}(\mathcal{T}_j) \right)_{21}\left( G_{ij}\otimes G_{ij}(\mathcal{T}_j) \right)^{-1}\right]T_{ij}^{-1}.
\end{equation}
The 1-arrow $g_{ij}$ is also compatible with the tensor structure. For any objects $V,W$ in $\text{BRep}(\mathscr{H}_{U_j})$, the transformation of tensor structure is
\begin{equation}
    \begin{array}{rccc}
J_{V,W}^{ij}: & g_{ij}(V)\otimes_i g_{ij}(W) & \longrightarrow & g_{ij}(V\otimes_j W) \\
   & v\otimes_i w & \longmapsto  & T_{ij}^{-1} \triangleright(v\otimes_j w).
\end{array}
\end{equation}
The 1-arrow is compatible not only with the tensor structure but also with the normalization condition. Using identity $\varepsilon_i \circ g_{ij}=\varepsilon_j$ and the normalization condition in (\ref{eq1.lemma1}), one obtains the normalization condition of twist transformation
\begin{equation}
    (\varepsilon_i \otimes \text{id})T_{ij} = (\text{id} \otimes \varepsilon_i)T_{ij}= 1.
\end{equation}
Therefore, the pair of gauge transformation and twist transformation induces a braided monoidal equivalence $g_{ij}$ between $\text{BRep}(\mathscr{H}_{U_i}|_{U_{ij}})$ and $\text{BRep}(\mathscr{H}_{U_j}|_{U_{ij}})$.

%%%%%%%%%%%%%%%%%%%%%%%%%%%%%%%%%%%%%%%%%%%%%%%%%%%%%%%%%%%%%%%%%%%%%%%%%%%%%%%%%%%
\section{Proof of Theorem \ref{theorem:stackification}}\label{appendix 3}
We prove the statement by constructing an explicit pseudoinverse to $\text{res}_{\mathcal{U}}$. The proof has four steps:
\begin{enumerate}
    \item Proving $\text{res}_{\mathcal{U}}:\mathscr{P}^+(U)  \longrightarrow \text{Des}_{\mathscr{P}^+}(\mathcal{U})$ is a pseudofunctor,
    \item Proving $\text{flat}_{\mathcal{U}}: \text{Des}_{\mathscr{P}^+}(\mathcal{U}) \longrightarrow \mathscr{P}^+(U) $ is a pseudofunctor,
    \item Proving $ \text{flat}_{\mathcal{U}}\circ \text{res}_{\mathcal{U}}
\cong
\text{id}_{\mathscr{P}^+(U)}$,
\item Proving $\text{res}_{\mathcal{U}}\circ \text{flat}_{\mathcal{U}}
\cong\text{id}_{\text{Des}{\mathscr{P}^+}(\mathcal{U})}$.
\end{enumerate}

\subsection{Proving $\text{res}_{\mathcal{U}}:\mathscr{P}^+(U)  \longrightarrow \text{Des}_{\mathscr{P}^+}(\mathcal{U})$ is a pseudofunctor}
Let $\mathcal A=
\left[ \mathcal V,\mathcal C_a,g_{ab},\alpha_{abc}
\right]_{\mathrm{ref},2\mathrm{eq}} \in \mathscr P^+(U) $ be represented by a $\mathscr P$-descent datum $ (\mathcal C_a,g_{ab},\alpha_{abc}) \in \text{Des}_{\mathscr P}(\mathcal V) $ on an admissible cover $\mathcal V=\{V_a\to U\}_{a\in A}$. For every member $U_i$ of the fixed cover $ \mathcal U=\{U_i\to U\}_{i\in I}, $ we define $x_i:=\mathcal A|_{U_i}=\left[\mathcal V|_{U_i},(\mathcal C_a,g_{ab},\alpha_{abc})|_{U_i}\right]_{\mathrm{ref},2\mathrm{eq}}
\in \mathscr P^+(U_i)$. Here $\mathcal V|_{U_i}=\{V_a\cap U_i\to U_i\}_{a\in A}$ is understood after replacing it by an admissible refinement if necessary. On the double overlap $U_{ij}=U_i\cap U_j$, we have $ x_i|_{U_{ij}} = \left[ \mathcal V|_{U_{ij}}, (\mathcal C_a,g_{ab},\alpha_{abc})|_{U_{ij}} \right]_{\mathrm{ref},2\mathrm{eq}}, $ and similarly $ x_j|_{U_{ij}} = \left[
\mathcal V|_{U_{ij}},
(\mathcal C_a,g_{ab},\alpha_{abc})|_{U_{ij}}
\right]_{\mathrm{ref},2\mathrm{eq}}$. Thus both restrictions are represented by the same $\mathscr P$-descent datum on the same restricted cover. Hence there is a canonical equivalence $ f_{ij}:x_j|_{U_{ij}}\longrightarrow x_i|_{U_{ij}}, $ represented by the identity $1$-morphism of $ (\mathcal C_a,g_{ab},\alpha_{abc})|_{U_{ij}} \in \text{Des}_{\mathscr P}(\mathcal V|_{U_{ij}})$. On the triple overlap $U_{ijk}=U_i\cap U_j\cap U_k$, the two composites $f_{ij}f_{jk}\qquad\text{and}\qquad f_{ik}$ are both represented by the identity $1$-morphism of the same restricted $\mathscr P$-descent datum $(\mathcal C_a,g_{ab},\alpha_{abc})|_{U_{ijk}}$. Therefore there is a canonical invertible $2$-morphism $\beta_{ijk}:f_{ij}f_{jk}\Longrightarrow f_{ik}$. This $2$-morphism is represented by the identity $2$-morphism after restriction to $\mathcal V|_{U_{ijk}}$. The coherence condition on quadruple overlaps is automatic. Indeed, on $U_{ijkl}=U_i\cap U_j\cap U_k\cap U_l$, the two possible composites from $f_{ij}f_{jk}f_{kl}$ to $ f_{il} $ are both represented by the identity $2$-morphism of the same restricted $\mathscr P$-descent datum $ (\mathcal C_a,g_{ab},\alpha_{abc})|_{U_{ijkl}}$. Hence the tetrahedral coherence condition holds. Therefore $
\text{res}_{\mathcal U}(\mathcal A):=(x_i,f_{ij},\beta_{ijk})$ is an object of $\text{Des}_{\mathscr P^+}(\mathcal U)$. The construction is compatible with $1$-morphisms. Let $F:\mathcal A\longrightarrow \mathcal B$ be a $1$-morphism in $\mathscr P^+(U)$. Choose representatives $\mathcal A=[\mathcal V,X]_{\mathrm{ref},2\mathrm{eq}},\mathcal B=[\mathcal V',Y]_{\mathrm{ref},2\mathrm{eq}}$. By the definition of the $2$-colimit, the morphism $F$ is represented, after passing to a common admissible refinement $\mathcal Z$ of $\mathcal V$ and $\mathcal V'$, by a genuine $1$-morphism $\widetilde F: X|_{\mathcal Z} \longrightarrow Y|_{\mathcal Z}$ inside $ \text{Des}_{\mathscr P}(\mathcal Z)$. Restricting this representative to each $U_i$ gives a $1$-morphism $F_i:\mathcal A|_{U_i}\longrightarrow\mathcal B|_{U_i}$ in $\mathscr P^+(U_i)$. On the double overlaps $U_{ij}$, the fact that
$\widetilde F$ is a morphism of $\mathscr P$-descent data gives invertible compatibility $2$-morphisms $\theta_{ij}: s_{ij}F_j \Longrightarrow F_i f_{ij}, $ where $f_{ij}$ and $s_{ij}$ are the canonical transition equivalences of $\text{res}_{\mathcal U}(\mathcal A)$ and $\text{res}_{\mathcal U}(\mathcal B)$, respectively. These compatibility $2$-morphisms satisfy the usual coherence condition on triple overlaps because $\widetilde F$ is a morphism of descent data. Hence $ (F_i,\theta_{ij}) $ is a $1$-morphism in $\text{Des}_{\mathscr P^+}(\mathcal U)$. Similarly, if $\eta:F\Longrightarrow G$ is a $2$-morphism in $\mathscr P^+(U)$, then after passing to a sufficiently fine common admissible refinement it is represented by a genuine $2$-morphism between representatives of $F$ and $G$. Restricting this representative to each $U_i$ gives compatible local $2$-morphisms $\eta_i:F_i\Longrightarrow G_i$. Thus $\eta$ determines a $2$-morphism in $\text{Des}_{\mathscr P^+}(\mathcal U)$. The construction is independent of the chosen representatives and refinements, because any two choices admit a further admissible common refinement, and the definition of $\mathscr P^+$ identifies representatives that agree after such a refinement. Moreover, identities and compositions are preserved up to the canonical coherence isomorphisms coming from restriction and common refinement. Therefore $\text{res}_{\mathcal U}:\mathscr P^+(U)\longrightarrow\text{Des}_{\mathscr P^+}(\mathcal U)$ is a well-defined pseudofunctor.

\subsection{Proving $\text{flat}_{\mathcal{U}}: \text{Des}_{\mathscr{P}^+}(\mathcal{U}) \longrightarrow \mathscr{P}^+(U) $ is a pseudofunctor}
Let $\mathcal X=(x_i,f_{ij},\beta_{ijk})
\in \text{Des}_{\mathscr P^+}(\mathcal U)$ be a $\mathscr P^+$-descent datum with respect to the cover $\mathcal U=\{U_i\to U\}_{i\in I}$. For each $i$, choose a representative of the object $x_i\in \mathscr P^+(U_i)$, say $x_i=\left[
\mathcal V_i,\mathcal C_{i\mu},g_{i\mu\nu}, \alpha_{i\mu\nu\rho}
\right]_{\mathrm{ref},2\mathrm{eq}}$, where $\mathcal V_i=\{V_{i\mu}\to U_i\}_{\mu\in M_i}$ is an admissible cover of $U_i$, and $(\mathcal C_{i\mu},g_{i\mu\nu},\alpha_{i\mu\nu\rho})$ is an object in $\text{Des}_{\mathscr P}(\mathcal V_i)$. Thus, for each $\mu$, $\mathcal C_{i\mu}\in \mathscr P(V_{i\mu})$, for each double overlap $V_{i\mu\nu}=V_{i\mu}\cap V_{i\nu}$, $g_{i\mu\nu}:\mathcal C_{i\nu}|_{V_{i\mu\nu}}\longrightarrow\mathcal C_{i\mu}|_{V_{i\mu\nu}}$, and for each triple overlap there is an invertible $2$-morphism $\alpha_{i\mu\nu\rho}:g_{i\mu\nu}g_{i\nu\rho}\Longrightarrow g_{i\mu\rho}$, satisfying the usual tetrahedral coherence condition on quadruple overlaps. We next represent the transition equivalences of the $\mathscr P^+$-descent datum by genuine morphisms between $\mathscr P$-descent data after refinement. On the double overlap $U_{ij}=U_i\cap U_j$, we have an equivalence $f_{ij}:x_j|_{U_{ij}}\longrightarrow x_i|_{U_{ij}}$ in $\mathscr P^+(U_{ij})$. Since $\mathscr P^+(U_{ij})=\mathop{2\text{-colim}}\limits_{\mathcal T\in\text{ACov}(U_{ij})}\text{Des}_{\mathscr P}(\mathcal T)$, the morphism $f_{ij}$ is represented, after passing to a sufficiently fine common admissible refinement $\mathcal T_{ij}\to \mathcal V_i|_{U_{ij}},\mathcal T_{ij}\to \mathcal V_j|_{U_{ij}}$, by an actual $1$-morphism of $\mathscr P$-descent data $\widetilde f_{ij}:(\mathcal C_{j\nu},g_{j\nu\lambda},\alpha_{j\nu\lambda\kappa})|_{\mathcal T_{ij}}\longrightarrow (\mathcal C_{i\mu},g_{i\mu\nu},\alpha_{i\mu\nu\rho})
|_{\mathcal T_{ij}}$ inside $\text{Des}_{\mathscr P}(\mathcal T_{ij})$. Similarly, on a triple overlap $U_{ijk}=U_i\cap U_j\cap U_k$, the coherence $2$-isomorphism $\beta_{ijk}:f_{ij}f_{jk}\Longrightarrow f_{ik}$ is represented, after passing to a sufficiently fine admissible common refinement $\mathcal T_{ijk}$, by an actual $2$-morphism $\widetilde\beta_{ijk}: \widetilde f_{ij}\widetilde f_{jk} \Longrightarrow \widetilde f_{ik} $ between morphisms of $\mathscr P$-descent data. Since the $\beta_{ijk}$ satisfy the tetrahedral coherence condition in
$\text{Des}_{\mathscr P^+}(\mathcal U)$, we refine once more, if necessary, so that the corresponding equality of $2$-morphisms is represented at the level of $\mathscr P$-descent data. Now choose a single admissible cover $\mathcal{W}=\{W_m\to U\}_{m\in M}$ which is a common refinement of all covers used above. In particular, each
$W_m$ is contained in some $U_{\sigma(m)}$, and $\mathcal{W}$ is chosen so that $\mathcal{W}|_{U_i},\mathcal{W}|_{U_{ij}},\mathcal{W}|_{U_{ijk}}$ are refinements of $\mathcal V_i,\mathcal T_{ij},\mathcal T_{ijk}$, respectively. Here $\sigma(m)\in I$ is a choice of index such that $W_m\subset U_{\sigma(m)}$. We now construct a candidate $\mathscr P$-descent datum $
\mathcal D=(\mathcal D_m,h_{mn},\gamma_{mnl}) $ on the cover $\mathcal{W}$. First, let $W_m\subset U_i$ with $i=\sigma(m)$. Since $\mathcal{W}|_{U_i}$
refines $\mathcal V_i$, there exists some $\mu$ such that
$W_m\subset V_{i\mu}$. Define $\mathcal D_m:=\mathcal C_{i\mu}|_{W_m}\in \mathscr P(W_m)$. Next, let $W_{mn}:=W_m\cap W_n$. Suppose $W_m\subset U_i, W_n\subset U_j$. We define a transition $1$-morphism $h_{mn}:\mathcal D_n|_{W_{mn}}\longrightarrow\mathcal D_m|_{W_{mn}} $ as follows. If $i=j$, then $W_m,W_n\subset U_i$, and $h_{mn}$ is defined as the restriction of the internal transition morphism $g_{i\mu\nu}$ of the local descent datum representing $x_i$. More precisely, if $W_m\subset V_{i\mu}$ and $W_n\subset V_{i\nu}$, set $h_{mn}:=g_{i\mu\nu}|_{W_{mn}}$. If $i\neq j$, then $W_{mn}\subset U_{ij}$, and $h_{mn}$ is defined to be the restriction to $W_{mn}$ of the appropriate local component of the representative $\widetilde f_{ij}:x_j|_{U_{ij}}\longrightarrow x_i|_{U_{ij}}$. Thus $h_{mn}$ is an actual $1$-morphism in $\mathscr P(W_{mn})$ from $\mathcal D_n|_{W_{mn}}$ to $\mathcal D_m|_{W_{mn}}$. Finally, on a triple overlap $W_{mnl}:=W_m\cap W_n\cap W_l$, with $W_m\subset U_i, W_n\subset U_j, W_l\subset U_k$, we define an invertible $2$-morphism $\gamma_{mnl}:h_{mn}h_{nl}\Longrightarrow h_{ml}$. If $i=j=k$, then all three local pieces lie inside the same $U_i$, and $\gamma_{mnl}$ is the restriction of the internal coherence $\alpha_{i\mu\nu\rho}:g_{i\mu\nu}g_{i\nu\rho}\Longrightarrow g_{i\mu\rho}$ of the descent datum representing $x_i$. If the indices are not all equal, then $\gamma_{mnl}$ is obtained from the appropriate local component of the representative $\widetilde\beta_{ijk}:\widetilde f_{ij}\widetilde f_{jk}\Longrightarrow\widetilde f_{ik}$, together with the compatibility $2$-morphisms expressing that each $\widetilde f_{ij}$ is a morphism of $\mathscr P$-descent data. It remains to verify that the candidate data $(\mathcal D_m,h_{mn},\gamma_{mnl})$ satisfy the descent coherence condition. On a quadruple overlap $W_{mnlq}:=W_m\cap W_n\cap W_l\cap W_q$, there are two composites from $h_{mn}h_{nl}h_{lq}$ to $h_{mq}$. If all four pieces lie inside the same $U_i$, then the equality of these two composites is exactly the tetrahedral coherence condition for the $\mathscr P$-descent datum representing $x_i$. In the mixed case, the two composites are the restrictions of the two sides of the tetrahedral coherence condition for the $\mathscr P^+$-descent datum $(x_i,f_{ij},\beta_{ijk})$. Since equality of $2$-morphisms in $\mathscr P^+$ is defined after passing to a sufficiently fine common refinement, and since $\mathcal{W}$ was chosen fine enough to represent this equality, the two composites agree on $W_{mnlq}$. Therefore, $(\mathcal D_m,h_{mn},\gamma_{mnl})\in \text{Obj}(\text{Des}_{\mathscr P}(\mathcal{W}))$. We define
\begin{equation}
\text{flat}_{\mathcal U}(\mathcal X) :=\left[\mathcal{W},\mathcal D_m,h_{mn},\gamma_{mnl}\right]_{\mathrm{ref},2\mathrm{eq}}\in \mathscr P^+(U).
\end{equation}
This definition is independent of the choices of representatives and refinements. Indeed, any two such choices admit a further admissible common refinement, and on that refinement the two flattened $\mathscr P$-descent data are equivalent. Hence they define the same object of $\mathscr P^+(U)$ by the definition of the $2$-colimit. The same construction applies to $1$ morphisms and $2$-morphisms. A $1$-morphism in $\text{Des}_{\mathscr P^+}(\mathcal U)$ is represented, after passing to suitable admissible refinements, by local morphisms between the chosen representatives of the $x_i$, compatible with the transition equivalences $f_{ij}$. Flattening these local representatives gives a $1$-morphism in $\mathscr P^+(U)$. Similarly, compatible local $2$-morphisms flatten to $2$-morphisms in $\mathscr P^+(U)$. The compatibility with identities and compositions follows from the compatibility of restriction with common refinement in the defining $2$-colimit. Therefore $\text{flat}_{\mathcal U}:\text{Des}_{\mathscr P^+}(\mathcal U)\longrightarrow\mathscr P^+(U)$ is a well-defined pseudofunctor.

\subsection{Proving $ \text{flat}_{\mathcal{U}}\circ \text{res}_{\mathcal{U}}
\cong \text{id}_{\mathscr{P}^+(U)}$}
Let $\mathcal{M}$ be an arbitrary object in $\mathscr{P}^+(U)$. Restricting $\mathcal{M}$ to the cover $\mathcal{U}$ gives the descent datum $\text{res}_\mathcal{U}(\mathcal{M})=(\mathcal{M}|_{U_i},f_{ij},\beta_{ijk})$. Flattening this datum amounts to restricting the original presentation of $\mathcal{M}$ to the common refinement of $\mathcal{V}$ and $\mathcal{U}$, for instance, to the cover by the nonempty intersections $V_a\cap U_i$, with further admissible refinement if needed. Thus, $\text{flat}_{\mathcal{U}}\text{res}_{\mathcal{U}}(\mathcal{M})$ is represented by the pullback of the original descent representation of $\mathcal{M}$ to a refinement. Since $\mathscr{P}^+(U)$ is defined by quotienting descent representations by refinement and $2$-equivalence, this pullback represents the same object as $\mathcal{M}$. This equivalence is obtained simply by gluing the identity equivalences on the local restrictions $\mathcal{M}|_{U_i}$ without any additional choice being required. Hence, there is a canonical equivalence
\begin{equation}
\varepsilon_{\mathcal{M}}:
\text{flat}_{\mathcal{U}}\text{res}_{\mathcal{U}}(\mathcal{M})
\mathop{\longrightarrow}\limits^{\cong}
\mathcal{M}.
\end{equation}
The same argument applies to $1$-morphisms and $2$-morphisms, because restriction followed by flattening only replaces representatives by their pullbacks to a common refinement. Therefore, the equivalences $\varepsilon_{\mathcal{M}}$ assemble into a pseudonatural equivalence
\begin{equation}
\text{flat}_{\mathcal{U}}\circ \text{res}_{\mathcal{U}}
\cong
\text{id}_{\mathscr{P}^+(U)}.
\end{equation}

\subsection{Proving $\text{res}_{\mathcal{U}}\circ \text{flat}_{\mathcal{U}}
\cong\text{id}_{\text{Des}{\mathscr{P}^+}(\mathcal{U})}$}
Let $\mathcal{Y}=(x_i,f_{ij},\beta_{ijk})\in\text{Des}_{\mathscr{P}^+}(\mathcal{U})$. By construction, $\text{flat}_{\mathcal{U}}(\mathcal{Y})$ is obtained by gluing representatives of the local objects $x_i$ using the transition equivalences $f_{ij}$ and the coherences $\beta_{ijk}$. Let $\mathcal{N}:=\text{flat}_{\mathcal{U}}(\mathcal{Y})\in \mathscr{P}^+(U)$. Restricting $\mathcal{N}$ back to $U_i$ gives an object $\mathcal{N}|_{U_i}\in \mathscr{P}^+(U_i)$. Since $\mathcal{N}$ was constructed by gluing the representatives of $x_i$, the object $\mathcal{N}|_{U_i}$ is equivalent to $x_i$, possibly after passing to an admissible refinement. Thus, there are 1-equivalences
\begin{equation}
e_i:\mathcal{N}|_{U_i}\longrightarrow x_i, \qquad \alpha_{ij}:e_i\circ g_{ij} \Longrightarrow f_{ij}\circ e_j
\end{equation}
in $\mathscr{P}^+(U_i)$. On double overlaps $U_{ij}$, these equivalences are compatible with the transition equivalences of $\text{res}_{\mathcal{U}}(\mathcal{N})$ and the given transition equivalences $f_{ij}$, precisely because the transitions of $\mathcal{N}$ were constructed from the $f_{ij}$. On triple overlaps, the compatibility with coherence follows from the fact that the coherences of $\mathcal{N}$ were constructed from the $\beta_{ijk}$. Let $\text{res}_\mathcal{U}(\mathcal{N})=\left(\mathcal{N}|_{U_i},g_{ij},\gamma_{ijk}\right)$, the following diagrams commute.

\begin{figure}[H]
\centering
\begin{minipage}{0.30\textwidth}
\centering
\begin{tikzpicture}[    
node distance=2.0cm and 2.4cm,     
obj/.style={font=\large},     
arrowlabel/.style={font=\small,inner sep=1pt},     
arrow/.style={-{To[length=7pt,width=9pt]},line width=0.6pt},twoarrow/.style={-{Implies[length=7pt,width=9pt]},double,double distance=1.2pt,line width=0.4pt}]

\node[obj] (A) at (0,0) {$\mathcal{N}|_{U_j}|_{U_{ij}}$};

\node[obj] (B) at (0,-2.0) {$\mathcal{N}|_{U_i}|_{U_{ij}}$};

\node[obj] (C) at (3.5,0) {$x_j|_{U_{ij}}$};

\node[obj] (D) at (3.5,-2.0) {$x_i|_{U_{ij}}$};

\node[obj] (E) at (2.7,-0.5) {};

\node[obj] (F) at (0.8,-1.5) {};

\draw[arrow] (A) -- node[left, arrowlabel] {$g_{ij}$} (B);

\draw[arrow] (A) -- node[above, arrowlabel] {$e_j$} (C);

\draw[arrow] (B) -- node[below, arrowlabel] {$e_i$} (D);

\draw[arrow] (C) -- node[right, arrowlabel] {$f_{ij}$} (D);

\draw[twoarrow] (F) -- node[below right, arrowlabel] {$\alpha_{ij}$} (E);

\end{tikzpicture}
\captionsetup{
    width=6cm,
    singlelinecheck=false
}
\end{minipage}
\hfill
\begin{minipage}{0.55\textwidth}
\centering
\begin{tikzpicture}[    
node distance=2.0cm and 2.4cm,     
obj/.style={font=\large},     
arrowlabel/.style={font=\small,inner sep=1pt},     
arrow/.style={-{To[length=7pt,width=9pt]},line width=0.6pt} ]

\node[obj](A) at (0,0) {$e_i\circ g_{ij}\circ g_{jk}$};

\node[obj](B) at (2.4,1.6) {$e_i\circ g_{ik}$};

\node[obj](C) at (4.8,0) {$f_{ik}\circ e_k$};

\node[obj](D) at (0,-1.8) {$f_{ij}\circ e_j\circ g_{jk}$};

\node[obj](E) at (4.8,-1.8) {$f_{ij}\circ f_{jk}\circ e_k$};

\draw[arrow] (A) -- node[above left, arrowlabel] {$1*\gamma_{ijk}$} (B);

\draw[arrow] (B) -- node[above right, arrowlabel] {$\alpha_{ik}$} (C);

\draw[arrow] (A) -- node[left, arrowlabel] {$\alpha_{ij}*1$} (D);

\draw[arrow] (D) -- node[below, arrowlabel] {$1*\alpha_{jk}$} (E);

\draw[arrow] (E) -- node[right, arrowlabel] {$\beta_{ijk}*1$} (C);

\end{tikzpicture}

\captionsetup{
    width=7.7cm,
    singlelinecheck=false
}
\end{minipage}
\end{figure}

\noindent Therefore, the family $(e_i)$ defines a $1$-equivalence in the bicategory $\text{Des}_{\mathscr{P}^+}(\mathcal{U})$:
\begin{equation}
\eta_{\mathcal{Y}}:
\text{res}_{\mathcal{U}}\text{flat}_{\mathcal{U}}(\mathcal{Y})
\longrightarrow
\mathcal{Y}.
\end{equation}
Again, the same reasoning applies to $1$-morphisms and $2$-morphisms. Hence the equivalences $\eta_{\mathcal{Y}}$ assemble into a pseudonatural equivalence
\begin{equation}
\text{res}_{\mathcal{U}}\circ \text{flat}_{\mathcal{U}}
\cong
\text{id}_{\text{Des}{\mathscr{P}^+}(\mathcal{U})}.
\end{equation}

In summary, we have constructed pseudofunctors
\begin{equation}
\text{res}_{\mathcal{U}}:
\mathscr{P}^+(U)
\longrightarrow
\text{Des}_{\mathscr{P}^+}(\mathcal{U}), \qquad \text{flat}_{\mathcal{U}}:
\text{Des}_{\mathscr{P}^+}(\mathcal{U})
\longrightarrow
\mathscr{P}^+(U),
\end{equation}
which are inverse to each other up to pseudonatural equivalence
\begin{equation}
    \text{flat}_{\mathcal{U}}\circ \text{res}_{\mathcal{U}}
\cong
\text{id}_{\mathscr{P}^+(U)}, \qquad \qquad \text{res}_{\mathcal{U}}\circ \text{flat}_{\mathcal{U}}
\cong
\text{id}_{\text{Des}{\mathscr{P}^+}(\mathcal{U})}.
\end{equation}
This proves $\text{res}_{\mathcal{U}}$ is a biequivalence
\begin{equation}
\text{res}_{\mathcal{U}}:
\mathscr{P}^+(U)
\mathop{\longrightarrow}\limits^{\cong}
\text{Des}_{\mathscr{P}^+}(\mathcal{U}).
\end{equation}
Since this holds for every cover $\mathcal{U}=\{U_i\to U\}_{i\in I}$, $\mathscr{P}^+$ satisfies effective 2-descent as a stackified prestack. Therefore, 
\begin{equation}
        \begin{array}{rccl}
\mathscr{P}^+: & \mathscr{S}_{\text{all}}^{\text{op}}& \longrightarrow & \textbf{BCatRe}, \\
 & U & \longmapsto  & \left\{\begin{array}{cl}
      \text{BRep}(\mathscr{H}_U)&\text{if } U \in \text{Loc}\mathcal{M},  \\
      \mathop{2\text{-colim}}\limits_{\mathcal{U}\in\text{ACov}(U)}{\text{Des}_\mathscr{P}(\mathcal{U})}&\text{if } U \in \text{Open}\mathcal{M}\backslash\text{Loc}\mathcal{M}
 \end{array}\right.
\end{array}
    \end{equation}
is a bicategory-valued stack.

\vspace{1cm}
\bibliographystyle{alpha-fr}
\bibliography{ref.bib}

 \end{document}